\documentclass[10pt, reqno]{amsart}

\usepackage{color}

\evensidemargin
\oddsidemargin 
\makeatletter\@addtoreset{equation}{section}\makeatother

\usepackage{amsmath, amsthm, latexsym, amssymb, graphics, bm, color}

\newcommand{\e}{\varepsilon}

\renewcommand{\phi}{\varphi}

\newcommand{\ssup}[1] {{\scriptscriptstyle{({#1}})}}
\newcommand{\PP}{\mathcal{P}_t^{\rm min}}

\newcommand{\Pt}{\text{\rm P}_{\mathcal{G}_t}}
\newcommand{\Et}{\text{\rm E}_{\mathcal{G}_t}}

\newcommand{\Pp}{\text{\rm P}_{\star}}
\newcommand{\Ep}{\text{\rm E}_{\star}}

\newcommand{\R}{\mathbb R}
\newcommand{\Z}{\mathbb Z}

\newcommand{\N}{\mathbb N}

\newcommand{\E}{\mathbb E}

\newtheorem{theorem}{Theorem}

\newtheorem{lemma}{Lemma}

\newtheorem{prop}{Proposition}

\def\1{{\mathchoice {1\mskip-4mu\mathrm l}      
{1\mskip-4mu\mathrm l}
{1\mskip-4.5mu\mathrm l} {1\mskip-5mu\mathrm l}}}

\renewcommand{\subsection}{\secdef \subsct\sbsect}
\newcommand{\subsct}[2][default]{\refstepcounter{subsection}
\vspace{0.15cm}
{\flushleft\bf \arabic{section}.\arabic{subsection}~\bf #1  }
\nopagebreak\nopagebreak}
\newcommand{\sbsect}[1]{\vspace{0.1cm}\noindent
{\bf #1}\vspace{0.1cm}}

\newcounter{remnr}

\newenvironment{remark}{\refstepcounter{remnr}
{\sf Remark~\arabic{remnr}.\ }\nopagebreak  }%
{\nopagebreak {\hfill{$\diamond$}}\\ }

\begin{document}

\title[Full Asymptotics for the heavy-tailed Pareto PAM]{Full asymptotics for the parabolic Anderson model with Pareto potential in the heaviest-tailed regime}

\author[Nadia Sidorova]{}

\maketitle

\centerline{\sc Nadia Sidorova\footnote{Department of Mathematics, University College London, Gower Street, London WC1 E6BT, UK, {\tt n.sidorova@ucl.ac.uk}.
} }

\vspace{0.4cm}

\vspace{0.4cm}

\begin{quote}
{\small {\bf Abstract:} The parabolic Anderson model is the Cauchy problem for the heat equation on the integer lattice with a random potential $\xi$. We consider the case where $\{\xi(z): z\in \Z^d\}$ are independent and identically distributed Pareto random variables with parameter~$\alpha$, and assume that the solution is initially localised at the origin. We establish the full asymptotic behaviour of the total mass of the solution as time tends to infinity in the heaviest-tailed regime $\alpha\in(d,2d)$. In particular, we find qualitatively different behaviour in dimension $d=1$ and in dimensions $d\ge 2$.
}

\end{quote}
\vspace{5ex}

{\small {\bf AMS Subject Classification:} 
Primary 60H25
Secondary 82C44, 60K37.

}
\vspace{4ex}





 \section{Introduction}
 
 Given a random potential $\xi: \Z^d\to \R$, the parabolic Anderson model is the Cauchy problem 
 with localised initial condition 
 
 \begin{equation}
 \label{pam}
\begin{aligned}
\partial_t u(t, z) &= \Delta u(t, z)+\xi(z)u(t, z),
 &&    (t, z)\in (0,\infty) \times \mathbb Z^d,\\
u(0,z) &= \mathbf 1_{\{0\}}(z),
 &&   z\in \mathbb Z^d,
\end{aligned}
\end{equation}
 where $\Delta$ is the discrete Laplacian acting on functions $f: \Z^d\to \R$ by 
 \begin{align*}
 (\Delta f)(z)=\sum_{|y-z|=1}(f(y)-f(z)), \qquad z\in \Z^d,
 \end{align*}
 with $|\cdot|$ denoting the standard $\ell_1$ distance. The parabolic Anderson model has its origins in the study of diffusion and reaction in disordered media and is closely related to the spectral theory of the Anderson Hamiltonian $\Delta+\xi$. 
Its most characteristic feature is intermittency, whereby the solution develops a small number of 
exceptionally high peaks carrying most of its mass. 
 The mathematical study of intermittency for the parabolic Anderson model was initiated by the seminal work of G\"artner and Molchanov~\cite{gm1} and has since attracted considerable attention in probability and mathematical physics. We refer to the monograph~\cite{k} for a comprehensive account of the model, its history and its connections with random walks in random potentials, branching processes, Anderson localisation and extreme value theory.
 
 \smallskip
 
 We assume that $\{\xi(z): z\in \Z^d\}$ are independent Pareto random variables with parameter $\alpha$, 
 that is, their distribution function is given by
 \begin{align*}
 F(x)=1-x^{-\alpha}, \qquad x\ge 1.
 \end{align*}
 We denote by $\text{P}$ and $\text{E}$ the corresponding probability and expectation and let 
 $\mu=\text{E}[\xi (z)]=\frac{\alpha}{\alpha-1}$. 
\smallskip

It follows from~\cite{gm1} that the solution of the parabolic Anderson model~\eqref{pam} exists provided 
that $\alpha>d$, and is given by the Feynman--Kac formula 
\begin{align}
\label{fc}
u(t,z)=\E\Big[\exp\Big\{\int_0^t \xi(X_s)ds\Big\}\1_{\{X_t=z\}}\Big], 
\qquad (t, z)\in (0,\infty) \times \mathbb Z^d,
\end{align}
where $(X_t)_{t\ge 0}$ is a continuous-time simple random walk on $\Z^d$ 
with generator $\Delta$
started at the origin, 
and $\E$ is the corresponding expectation. We denote by 
 \begin{align}
 \label{fct}
 U(t)=\sum_{z\in \Z^d}u(t, z)
 =\E\Big[\exp\Big\{\int_0^t \xi(X_s)ds\Big\}\Big], \qquad t>0, 
 \end{align}
 the total mass of the solution. 
\smallskip

One of the central questions in the study of the parabolic Anderson model  is the large-time behaviour of the total mass $U(t)$ or, equivalently, of 
\begin{align*}
L(t)=\frac 1 t \log U(t).
\end{align*}
In the case of Pareto potential, both the almost sure and weak asymptotics of $L(t)$ were studied in~\cite{hms}. Owing to the heavy-tailed nature of the Pareto distribution, already the leading term of the logarithmic weak asymptotics turned out to be random. More precisely, it was shown that
\begin{align}
\label{a}
\frac{L(t)}{a_t}\Rightarrow \Upsilon \qquad \text{with}\quad a_t=\Big(\frac{t}{\log t}\Big)^{\frac{d}{\alpha-d}},
\end{align}
where $\Upsilon$ is a random variable with an explicitly known positive density on $(0,\infty)$.
\smallskip


It was then
proved in~\cite{klms} that the parabolic Anderson model with Pareto potential exhibits 
complete localisation. Specifically, for each $t$ there exists a single site $Z_t^*\in \Z^d$ 
that carries an overwhelming proportion of the total mass of the solution, that is,
\begin{align}
\label{comploc}
\frac{u(t, Z_t^*)}{U(t)}\to 1\qquad \text{in probability}.
\end{align}
Moreover, it was shown in~\cite{klms} that localisation holds in the stronger sense of almost-sure convergence, for which two sites are necessary to accommodate transitions between successive localisation sites.
More precisely, for each $t$ there exist two sites $Z_t^*, \hat Z_t^*\in \Z^d$ such that
\begin{align}
\label{comploc2}
\frac{u(t, Z_t^*)+u(t,\hat Z_t^*)}{U(t)}\to 1\qquad \text{almost surely}.
\end{align}
The one-point localisation result~\eqref{comploc} first appeared in the earlier preprint~\cite{kms}; since it was subsequently included in~\cite{klms}, we cite the latter throughout the paper. Related heavy-tailed localisation phenomena have also been studied for the parabolic Anderson model on random trees in~\cite{ap} and for branching random walks in Pareto environments in~\cite{or1,or2}.
In lighter-tailed regimes, one-point localisation was later established for Weibull potentials in~\cite{fm, st}. For double-exponential potentials, point localisation no longer holds; instead, the mass concentrates in a bounded neighbourhood of a single site~\cite{bks}.
\smallskip

The localisation picture~\eqref{comploc} and~\eqref{comploc2} also led to a much more precise description of the asymptotic behaviour of the total mass $U(t)$. It was shown in~\cite{klms} that the localisation sites $Z_t^*$ and $\hat Z_t^*$ 
are the first and second maximisers of the functional
\begin{align}
\label{phi}
\Psi_t^* (z)=\Big[\xi(z)-\frac{|z|}{t}\log \xi(z)+\frac{\log N(z)}{t}\Big]\1_{\{t\xi(z)\ge |z|\}}, \qquad (t, z)\in (0,\infty) \times \Z^d,
\end{align}
where $N(z)$ denotes the number of paths of minimal length $|z|$ from the origin to $z$ in $\Z^d$. This reflects the trade-off faced by the random walk $(X_t)$ in the Feynman--Kac formula~\eqref{fc}: spending time at sites with large potential yields a greater exponential reward, but reaching such sites and staying there incurs an entropic cost. It was also shown that 
\begin{align}
\label{lt}
L(t)=(1+o(1))\Psi_t^*(Z_t^*)\qquad\text{in probability},
\end{align}
thus giving a much sharper description of the leading term of the asymptotics.
The joint scaling limit of the localisation site $Z_t^*$ and 
of the term $\Psi_t^*(Z_t^*)$ describing $L(t)$
was found to be 
\begin{align}
\label{r}
\Big(\frac{Z_t^*}{r_t}, \frac{\Psi_t^*(Z_t^*)}{a_t}\Big)
\Rightarrow (X, \Upsilon), 
\end{align}
where 
\begin{align*}
r_t=\Big(\frac{t}{\log t}\Big)^{\frac{\alpha}{\alpha-d}}
\end{align*}
and  $(X, \Upsilon)$ is a random variable with an explicitly known positive density on $\R^d\times (0,\infty)$.
Note that the relation 
\begin{align}
\label{ra}
a_t=r_t^{d/\alpha}
\end{align}
between $r_t$ and $a_t$
reflects the fact that the maximum of $n$ i.i.d.\ Pareto random variables with parameter $\alpha$ is of order $n^{1/\alpha}$, while the number of lattice points in a $d$-dimensional ball grows proportionally to the $d$th power of its radius.
\smallskip

While the methods developed to prove~\eqref{lt} and~\eqref{r} provided a good description of the leading term of $L(t)$, they left the remainder essentially untouched, determining it only up to an error of order $o(a_t)$.
The aim of this paper is to understand this remainder and thereby derive the full asymptotic expansion of the total mass in the heaviest-tailed regime $\alpha\in(d,2d)$.
Recall that the case $\alpha\le d$ is ruled out because the solution explodes.
The case $\alpha\geq 2d$ presents additional difficulties, discussed below, and lies beyond the scope of the present methods.
\smallskip

The first step in this direction was taken in~\cite{mps1,mps2} for the partially symmetric parabolic Anderson model, where instead of being fully i.i.d., the potential $\xi$ is equal, with probability $p$, at pairs of sites symmetric with respect to the origin, and independent otherwise. The analysis was restricted to the one-dimensional setting, where a phase transition at $\alpha=2$ was identified. It was shown that for $\alpha\ge2$, the solution remains localised at a single site despite the partial symmetry, whereas for $\alpha\in(1,2)$ the symmetry manifests itself, with two symmetric sites carrying a non-negligible proportion of the total mass. Distinguishing between these two sites requires comparing their masses, and since the leading-order asymptotics are identical, this comparison necessarily relies on understanding the lower-order terms.
\smallskip

Our first result, Theorem~\ref{th_1}, adapts some of the ideas developed in~\cite{mps1,mps2} to derive the full asymptotic expansion of $U(t)$ in the one-dimensional heaviest-tailed regime $\alpha\in(1,2)$. 
\smallskip

\begin{theorem} 
\label{th_1}
Let $d=1$ and $\alpha\in (1,2)$. As $t\to\infty$, 
\begin{align}
\label{asymp}
U(t)=\exp\Big\{t\Psi_t^*(Z_t^*)-2t+\mu\frac{|Z_t^*|}{\xi(Z_t^*)} + 
\chi_t\Big\},
\end{align}
with
\begin{align}
\label{conv4}
\mathcal{L}\Big(\frac{Z^*_t}{r_t}, \frac{\Psi_t^*(Z_t^*)}{a_t},  \frac{\xi(Z^*_t)}{a_t}, \chi_t\Big)\Rightarrow 
\mathcal{L}(X, \Upsilon, \Upsilon + q |X|, \chi),
\end{align}
where 
\begin{align*}
q=\frac{d}{\alpha-d}
\end{align*}
and 
the triple $(X, \Upsilon, \chi)$ has a positive density on $\R\times (0,\infty)\times \R$. 
\end{theorem}
\smallskip

\begin{remark} 
We know from~\eqref{r} that the first two coordinates in~\eqref{conv4} converge and that $(X,\Upsilon)$ has a positive density. Extending this to the third coordinate is simple, due to the relationship between $\Psi_t^*$ and $\xi$.
The new content of~\eqref{conv4} is the identification of the finite-order random term $\chi_t$. In the proof of Theorem~\ref{th_1}, we describe the limiting distribution $\chi$, conditionally on $(X,\Upsilon)$, in terms of a time-inhomogeneous compensated Poisson process.\end{remark}


The one-dimensional techniques, however, do not extend to higher dimensions because of the fundamentally different structure of the path space. In dimension one, there is a unique shortest path connecting the origin to the localisation site $Z_t^*$. In contrast, in higher dimensions there are many such paths, so the central limit theorem arguments used in dimension one no longer apply. Instead, the problem becomes one of analysing a heavy-tailed random polymer; see, for example,~\cite{al, nikos}.
\smallskip

Our second and main result, Theorem~\ref{th_d}, develops a new approach to obtain the full asymptotic expansion of $U(t)$ in arbitrary dimension $d\ge 2$, in the regime $\alpha\in(d,2d)$.

\begin{theorem} 
\label{th_d}
Let $d\ge 2$
and $\alpha \in (d, 2d)$. As $t\to \infty$
\begin{align}
\label{asymp_d}
U(t)=\exp\Big\{t\Psi_t^*(Z_t^*)-2dt+\mu\frac{|Z_t^*|}{\xi(Z_t^*)} + o(1)\Big\}
\end{align}
in probability. 
\end{theorem} 

\begin{remark} 
As our results show, in all dimensions the first three terms in the logarithmic asymptotics are of orders
$t a_t\gg t\gg \frac{t}{\log t}$,
whereas the final term is finite and random in dimension one, but negligible in higher dimensions. This difference is largely due to a random-polymer averaging effect arising from the large number of shortest paths.
\end{remark}



Our approach is tailored to the regime $\alpha\in(d,2d)$ and does not 
extend to $\alpha\ge 2d$. The first obstacle is that, in the lighter-tailed Pareto regime, the total mass $U(t)$ receives non-negligible contributions not only from the shortest paths but from many longer paths as well. The second is that the slower growth of the extreme values of the potential $\xi$ makes path intersections play a significantly more prominent role in the random polymer picture.
\bigskip
 

\section{Notation and proof strategy}

First of all, we introduce a simpler version 
\begin{align*}
\Psi_t(z)=\xi(z)-\frac{|z|}{t}q\log t, \qquad (t,z)\in(0,\infty)\times\Z^d,
\end{align*}
of the functional $\Psi_t^*$, which is more convenient to work with, 
and we denote its maximiser by $Z_t$. Further, we denote
\begin{align*}
X_t=\frac{Z_t}{r_t}, \qquad Y_t=\frac{\xi(Z_t)}{a_t}
\qquad \text{and}\qquad \Upsilon_t=\frac{\Psi_t(Z_t)}{a_t}.
\end{align*}
Throughout the paper we use several scaling limit results which, in particular, quantify the close relationship between $\Psi_t^*$ and $\Psi_t$, and between their respective maximisers $Z_t^*$ and $Z_t$. These results are obtained using standard point process techniques. The arguments are similar to those in earlier work, including~\cite{klms}, and we therefore postpone their proofs to the Appendix.
In Proposition~\ref{zz} we show that $Z_t=Z_t^*$ with probability tending to one. In particular, the complete localisation result~\eqref{comploc} can equivalently be written as
\begin{align}
\label{comploc1}
\frac{u(t,Z_t)}{U(t)}\to1\qquad\text{in probability}.
\end{align}
In Proposition~\ref{p_new}  we show that
$(X_t, \Upsilon_t, Y_t)\Rightarrow (X, \Upsilon, \Upsilon+q|X|)$, which we then use to obtain the same scaling limit for the first three coordinates in~\eqref{conv4}. 
We note, however, that  
the term $\Psi_t^*(Z_t^*)$ in Theorems~\ref{th_1} and~\ref{th_d} cannot be replaced by 
$\Psi_t(Z_t)$, since the two agree only to first order.
\smallskip

The starting point of our approach is the Feynman--Kac formula~\eqref{fct}, which we interpret as the sum over paths in $\Z^d$, each weighted by its contribution after averaging over the holding times of the continuous-time random walk $(X_t)$. By the localisation result~\eqref{comploc1}, it suffices to consider paths connecting the origin to the localisation site $Z_t$. We derive an explicit expression for the contribution of each such path in terms of the potential values along it and show that, in the regime $\alpha\in(d,2d)$, only the shortest paths contribute to the asymptotics. In dimension one, where there is only one such path, we use characteristic functions, 
guided by the point process intuition, to identify the finite-order random fluctuations. In higher 
dimensions, we use $L^2$-estimates based on an analysis of intersections and self-intersections of the paths.
\smallskip

For technical reasons, we also introduce some auxiliary positive scaling functions $f_t\to 0$ and $g_t\to\infty$ which can be thought of as arbitrarily slowly decaying or growing. We shall need these scales to satisfy 
\begin{align*}
g_t, 1/f_t=O(\log\log\log t). 
\end{align*}

The localisation~\eqref{comploc} and therefore~\eqref{comploc1} was proved in~\cite{klms} by showing a stronger result, namely, that 
\begin{align}
\label{u0u}
\frac{U_0(t)}{U(t)}\to 1\qquad\text{in probability},  
\end{align}
where 
\begin{align}
\label{u0}
U_0(t)=\E\Big[\exp\Big\{\int_0^t \xi(X_s)ds\Big\}\1_{\{X_t=Z_t, J_t<R_t\}}\Big], 
\qquad t>0, 
\end{align}
with $R_t=|Z_t|(1+f_t)$ and $J_t$ being the number of jumps of the random walk $(X_s)$ by time $t$.
This means that only reasonably short trajectories of the random walk ending at $Z_t$ make a non-negligible contribution to $U(t)$.
\smallskip

Denote by 
\begin{align*}
\mathcal{P}_{all}=\big\{y=(y_0, \dots, y_{\ell})\in (\Z^d)^{\ell+1}: \ell\in \N_0, y_0=0, |y_i-y_{i-1}|=1
\text{ for all }1\le i\le \ell\big\}
\end{align*}
the set of all paths on $\Z^d$ starting at the origin. 
For each $y\in \mathcal{P}_{all}$, denote by $\ell(y)$ its length (counted as the number of edges). 
Denote by 
\begin{align*}
\mathcal{P}_t=\big\{y\in  \mathcal{P}_{all}: y_{\ell(y)}=Z_t, \ell(y)<R_t\big\}
\end{align*}
the set of relevant paths according to~\eqref{u0} and denote by  
\begin{align*}
\PP=\big\{y\in  \mathcal{P}_t: \ell(y)=|Z_t|\big\}
\end{align*} 
the set of shortest paths.
\smallskip  

Further, denote by $(\tau_i)_{i\in\mathbb N_0}$ the sequence of holding times of the continuous-time random walk $(X_t)$, each exponentially distributed with rate $2d$, and by 
\begin{align*}
P(t, y)=\Big\{X_{\tau_0+\cdots+\tau_{i-1}}=y_i\text{ for all }1\le i\le \ell(y), 
t-\tau_{\ell(y)}\le \tau_0+\cdots+\tau_{\ell(y)-1}<t\Big\}, \qquad y\in\mathcal{P}_{all},
\end{align*}
the event that the random walk has the trajectory $y$ up to time $t$. Let
\begin{align*}
U(t,y)=\E \Big[\exp\Big\{\int_0^t\xi(X_s)ds\Big\}\1_{P(t,y)}\Big]
\end{align*}
be the contribution to $U(t)$ coming from the random walk trajectories following the path $y$.
According to~\eqref{u0} we have
\begin{align}
\label{deco}
U_0(t)=\sum_{y\in \mathcal{P}_t}U(t,y).
\end{align}

Generalising (2.2) in~\cite{mps1} from $\Z$ to $\Z^d$ by using the jump rate $2d$ instead of $2$, we observe that 
\begin{align}
\label{uty}
U(t,y)=e^{-2dt}I_{\ell(y)}(t; \xi(y_0), \dots, \xi(y_{\ell(y)})), 
\end{align}
where the function $I$ is defined by
\begin{align*}
I_n(t; a_0, \dots, a_n)=e^{ta_n}\int_{\R_+^n}\exp\Big\{\sum_{i=0}^{n-1}x_i(a_i-a_n)\Big\}
\1\Big\{\sum\limits_{i=0}^{n-1}x_i<t\Big\}dx_0\dots dx_{n-1},
\end{align*}
for each $t>0$, $n\in \N_0$, and $a_0, \dots, a_n\in \R$. In particular, $I_0(t; a_0)=e^{ta_0}$. 
It was shown in~\cite[Lemma 3.8]{mps1} that the function $I$ possesses a remarkably symmetric structure. In Proposition~\ref{p_shortlong} we exploit this structure to bound $U(t,y)$ in terms of the product
\begin{align}
\label{products}
\prod_{\substack{i=0\\ y_i\neq Z_t}}^{\ell(y)}\frac{1}{\xi(Z_t)-\xi(y_i)}, 
\end{align}
the length of the path, and the number of visits to $Z_t$. In particular, each extra step of the path incurs a penalty factor of order $\frac{1}{\xi(Z_t)}$, 
while each additional visit to $Z_t$ contributes a factor of order $\frac{t^m}{m!}$, where $m$ is the number of extra visits. These penalties must, however, be balanced against the combinatorial growth in the number of such paths. 
\smallskip

To compare the path-counting entropy with the penalty factors above, we partition the paths according to three parameters. For any $y\in\mathcal{P}_{all}$, let
\begin{align*}
m_t(y)=|\{0\le i\le \ell(y): y_i=Z_t\}|-1
\end{align*}
be the number of repeated visits to $Z_t$, let
\begin{align*}
w_t(y)=\frac{\min\{i: y_i=Z_t\}-|Z_t|}{2}
\end{align*}
be half the number of extra steps taken before the first visit to $Z_t$, and let
\begin{align*}
s_t(y)=|\{|Z_t|+2w_t(y)<i\le \ell(y): y_i\neq Z_t\}|
\end{align*}
be the number of sites visited after the first visit to $Z_t$ that are different from $Z_t$. 
Observe that 
\begin{align*}
s_t(y)\ge m_t(y)\qquad\text{and}\qquad \ell(y)=|Z_t|+m_t(y)+2w_t(y)+s_t(y).
\end{align*}
Further, observe that since $\ell(y)<R_t$, we have  
\begin{align}
\label{mwsR}
m_t(y)+2w_t(y)+s_t(y)<|Z_t|f_t
\end{align}
for any $y\in \mathcal{P}_{t}$.
\smallskip

For each $t>0$ and $m, w, s\in \N_0$ such that $s\ge m$ we denote
\begin{align*}
\mathcal{P}_{t, m,w,s}=\Big\{y\in \mathcal{P}_t: m_t(y)=m, w_t(y)=w, s_t(y)=s\Big\}. 
\end{align*}
In Lemma~\ref{l_paths} we give an upper bound for the size of each $\mathcal{P}_{t, m,w,s}$.
\smallskip

In order to understand the products~\eqref{products}, we write them, with high probability, in the form
\begin{align*}
\prod_{\substack{i=0\\ y_i\neq Z_t}}^{\ell(y)}\frac{1}{\xi(Z_t)-\xi(y_i)}
=\xi(Z_t)^{-\ell(y)+m_t(y)}\exp\Big\{\beta_t\sum_{\substack{i=0\\ y_i\neq Z_t}}^{\ell(y)}
\zeta_t(y_i)\Big\},
\end{align*}
where
\begin{align*}
\beta_t=\frac{1}{\xi(Z_t)}
\end{align*}
and
\begin{align}
\label{zeta}
\zeta_t(z)=-\xi(Z_t)\log \Big(1-\frac{\xi(z)}{\xi(Z_t)}\Big)\1_{\{\xi(z)<(1-pf_t)\xi(Z_t)\}},
\end{align}
with $p=0$ if $d=1$ and $p=1$ if $d\ge 2$.
Conditionally on $(Z_t,\xi(Z_t))$, this representation places the problem naturally in the framework of random polymers: the path $y$ plays the role of the polymer, $\beta_t$ that of the inverse temperature, and the variables $\zeta_t(z)$ that of the random environment. Since $\zeta_t(z)\approx \xi(z)$ when $\xi(z)$ is small relative to $\xi(Z_t)$, the exponential term can be viewed as a polymer weight in an approximately Pareto environment.
\smallskip

For each $t$, let $\mathcal{G}_t$ be the $\sigma$-algebra generated by $Z_t$ and $\xi(Z_t)$, and denote 
the conditional probability and conditional expectation with respect to $\mathcal{G}_t$ by $\Pt$ 
and $\Et$, respectively.  
\smallskip

Section~\ref{s_1dim} is devoted to the one-dimensional case. In Proposition~\ref{p_u1dim}, we show that $U(t)$ can be described in terms of the unique shortest path connecting the origin to $Z_t$. This reduces the random polymer picture to a $\mathcal{G}_t$-conditional central limit theorem problem for the single sum 
\begin{align*}
\sum\limits_{|z|<|Z_t|, zZ_t\ge 0}\zeta_t(z), 
\end{align*}
which we analyse through its $\mathcal{G}_t$-conditional characteristic function. Proposition~\ref{p_psi} provides a representation of this characteristic function in terms of a compensated Poisson process, allowing us to identify the limiting behaviour of the fluctuation term $\chi_t$ in~\eqref{asymp}.
\smallskip

The multidimensional case is treated in Sections~\ref{s_prelim} and~\ref{s_multi}. In Proposition~\ref{p_product}, we centre and normalise the polymer weights by introducing
\begin{align}
\label{ht}
H_t(y)=\prod_{\substack{i=0\\ y_i\neq Z_t}}^{\ell(y)}\exp\Big\{\beta_t \big[\zeta_t(y_i)-\mu_t(y_i)\big]
-\lambda_{t,y_i}(\beta_t)\Big\}-1, 
\qquad y\in \mathcal{P}_{all}, 
\end{align}
where 
\begin{align*}
\mu_t(z)=\Et [\zeta_t(z)], \qquad z\in \Z^d,
\end{align*}
and 
\begin{align*}
\lambda_{t, z}(s)=\log \Et \exp\Big\{s\big[\zeta_t(z)-\mu_t(z)\big]\Big\},
\qquad s\in \R, z\in \Z^d.
\end{align*}
Thus each path $y$ essentially carries the polymer weight $1+H_t(y)$. The deterministic part $1$ is potentially problematic because of the large number of admissible paths, while the fluctuation term $H_t(y)$ requires control of correlations between paths in order to obtain suitable $L^2$-estimates.
\smallskip

In Proposition~\ref{p_xixi}, we separate the contributions from shortest and non-shortest paths. Proposition~\ref{p_xi} treats the deterministic part $1$ for non-shortest paths, showing that, when $\alpha\in(d,2d)$, the combinatorial growth of the spaces $\mathcal{P}_{t,m,w,s}$ is dominated by the penalties incurred by such paths.
\smallskip

The fluctuation terms $H_t(y)$ are treated in a similar way for shortest and non-shortest paths, although in the latter case the sizes of the path spaces and the corresponding penalty terms must again be taken into account. In Proposition~\ref{p_c}, the correlations between $H_t(y)$ and $H_t(\hat y)$ are controlled through intersections of the corresponding paths and are shown to be sufficiently small in the regime $\alpha\in(d,2d)$. This implies that these terms do not contribute significantly to the asymptotics of $U(t)$.
\smallskip

As a result, the only non-negligible contribution comes from the deterministic part $1$ of the polymer weights along the shortest paths, which yields the asymptotics stated in Theorem~\ref{th_d}.
\bigskip



Throughout the paper we use several typical events which occur with overwhelming probability. 
According to~\eqref{u0u} there is an event $\mathcal{E}_t^{\text{loc}}$ such that 
\begin{align}
\label{comploc0}
\text{\rm P}\big(\mathcal{E}_t^{\text{loc}}\big)\to 1\qquad\text{and}\qquad  
\Big(1-\frac{U_0(t)}{U(t)}\Big)\1_{\mathcal{E}_t^{\text{loc}}}\to 0\qquad\text{almost surely}. 
\end{align}
Further, we define 
\begin{align*}
\mathcal{E}_t^{\ssup 1}=
\Big\{f_t<|X_t|<g_t, f_t<Y_t< g_t, 
\min_{1\le i\le d}\frac{|X_{t,i}|}{|X_t|}>f_t,
1-q\frac{|X_t|}{Y_t}>f_t \Big\}
\end{align*}
where $X_t=(X_{t,1}, \dots, X_{t,d})$,  and
\begin{align*}
\mathcal{E}_t^{\ssup 2}
=\Big\{\max_{\substack{|z|<R_t\\ z\neq Z_t}}\frac{\xi(z)}{\xi(Z_t)}<1-f_t\Big\},
\end{align*}
which contain useful typical properties of the random field. 
We consider those two events separately since
$\mathcal{E}_t^{\ssup 1}\in \mathcal{G}_t$
but $\mathcal{E}_t^{\ssup 2}\not\in \mathcal{G}_t$. We will show in Proposition~\ref{events}
in the Appendix that $\text{\rm P}(\mathcal{E}_t^{\ssup 1})\to 1$ and $\text{\rm P}(\mathcal{E}_t^{\ssup 2})\to 1$. 
\smallskip

Throughout the paper, we use the asymptotic notation $o(\cdot)$ and $O(\cdot)$
uniformly over the underlying probability space.
Specifically, for positive random functions
$a_t$ and $b_t$ we write $a_t=o(b_t)$ 
if there exists a deterministic function $c_t\to 0$ such that $\frac{a_t}{b_t}<c_t$ eventually for all $t$
almost surely.  Similarly, we write $a_t=O(b_t)$ 
if there exists a deterministic constant $c$ such that $\frac{a_t}{b_t}<c$ eventually for all $t$
almost surely. 
\smallskip

\bigskip


\section{Path contributions and path counting}

We begin by representing the contribution $U(t,y)$ of a path $y$ to the total mass in terms of a product of potential differences along the path, together with factors accounting for repeated visits to $Z_t$. Since the shortest paths will turn out to provide the main contribution to $U(t)$, we need precise asymptotics for their total contribution, although not for each path individually; for non-shortest paths, a sharp upper bound is sufficient. 

\begin{prop} 
\label{p_shortlong}
$\phantom{a}$

\begin{itemize}
\item[(1)] Almost surely on the event $\mathcal{E}_t^{\rm loc}\cap\mathcal{E}_t^{\ssup 1}\cap \mathcal{E}_t^{\ssup 2}$,  as $t\to \infty$
\begin{align}
\label{short}
\sum_{y\in\PP}U(t,y)=e^{t\xi(Z_t)-2dt}\sum_{y\in \PP}
\prod_{j=0}^{|Z_t|-1}\frac{1}{\xi(Z_t)-\xi(y_j)} + o\big(U(t)\big).
\end{align}
\item[(2)] 
Almost surely on the event $\mathcal{E}_t^{\ssup 2}$, for all sufficiently large $t$
\begin{align}
\label{long}
U(t,y)\le e^{t\xi(Z_t)-2dt}\frac{t^m}{m!}\prod_{\substack{i=0\\ y_i\neq Z_t}}^{\ell(y)}
\frac{1}{\xi(Z_t)-\xi(y_i)}   \phantom{aal}
\end{align}
for any $y\in \mathcal{P}_{t, m,w,s}$ for each $(m,w,s)$.
\end{itemize}
\end{prop}

\begin{proof} 
(1)
Consider $y\in\PP$. Since $y$ has no self-intersections, we can 
apply Lemma 3.8 from~\cite{mps1} to the function $I$ in~\eqref{uty}, which yields
\begin{align}
U(t,y)
=e^{-2dt}\Big[&e^{t\xi(Z_t)}\prod_{j=0}^{|Z_t|-1}\frac{1}{\xi(Z_t)-\xi(y_j)}\notag\\
&-\sum_{i=0}^{|Z_t|-1}I_i(t; \xi(y_0), \dots, \xi(y_i))\prod_{j=i}^{|Z_t|-1}\frac{1}{\xi(Z_t)-\xi(y_j)}\Big].
\label{t5}
\end{align}
Now it suffices to show that the contribution from the subtracted term is negligible.  Observe that using again~\eqref{uty} we can rewrite 
\begin{align*}
I_i(t; \xi(y_0), \dots, \xi(y_i))=e^{2dt} U(t, w_i(y)),
\end{align*}
where $w_i(y)$ is the truncation of $y$ to its first $i$ edges. Using 
$\xi(Z_t)-\xi(y_j)\ge f_t\xi(Z_t)>d$ for all $j$ on the event 
$\mathcal{E}_t^{\ssup 1}\cap \mathcal{E}_t^{\ssup 2}$ 
we obtain,  
summing over all shortest paths, 
for the subtracted term of~\eqref{t5}
\begin{align*}
\sum_{y\in\PP}
&e^{-2dt}
\sum_{i=0}^{|Z_t|-1}I_i(t; \xi(y_0), \dots, \xi(y_i))\prod_{j=i}^{|Z_t|-1}\frac{1}{\xi(Z_t)-\xi(y_j)}
\le \sum_{y\in\PP} 
\sum_{i=0}^{|Z_t|-1}d^{\, i - |Z_t|} U(t, w_i(y)).
\end{align*} 
Observe that the same $i$-truncated path  can correspond to no more than $d^{|Z_t|-i}$
paths $y\in\PP$. According to~\eqref{fct}, ~\eqref{u0} and~\eqref{comploc0} on the event $\mathcal{E}_t^{\rm loc}$, this yields 
\begin{align*}
\sum_{y\in\PP} 
\sum_{i=0}^{|Z_t|-1}d^{\, i -|Z_t|}U(t, w_i(y))
\le \sum_{i=0}^{|Z_t|-1}\sum_{w: \ell(w)=i}U(t, w)
\le U(t)-U_0(t)=o(U(t)),
\end{align*} 
which completes the proof of~\eqref{short}. 
\smallskip

(2) Let us consider $y\in \mathcal{P}_{t, m,w,s}$ and again use~\eqref{uty}. It was shown in~\cite[Lemma 3.8]{mps1} that $I$ is symmetric with respect to the values $\xi(y_0), \dots, \xi(y_{\ell(y)})$, which allows us to re-order them in such a way that the values $\xi(Z_t)$ occupy the last $m+1$ places. Since 
$\xi(Z_t)>\xi(y_i)$ whenever $y_i\neq Z_t$ on the event $\mathcal{E}_t^{\ssup 2}$, now the result follows from~\cite[Lemma 3.9]{mps1} with $n=\ell(y)-m$, $k=n$ and $i=m$, $a_0, \dots, a_{n-1}$ being the values of $\xi$ along $y$ except when it visits $Z_t$ and $a_n, \dots, a_{n+m}=\xi(Z_t)$, and 
from $I_0(t, a_n)=e^{ta_n}$. 
\end{proof}


In the following lemma we establish an upper bound  on the size of the path spaces 
$\mathcal{P}_{t,m,w,s}$ relative to the size of $\mathcal{P}_t^{\rm min}$. 

\begin{lemma} 
\label{l_paths}
Almost surely on the event $\mathcal{E}_t^{\ssup 1}$
\begin{align*}
|\mathcal{P}_{t, m, w, s}|\le N(Z_t) \Big(\frac{2d|Z_t|}{f_t}\Big)^w\frac{(2d)^s}{w!}
\end{align*}
for all $t$ and all $(m,w,s)$ .
\end{lemma}

\begin{proof}  Denote $Z_t=(Z_{t,1}, \dots, Z_{t,d})$ and observe that  
\begin{align}
\label{sh}
N(Z_t)={|Z_t| \choose |Z_{t,1}|, \dots, |Z_{t,d}|}.
\end{align}
First, let us estimate $|\mathcal{P}_{t,0,w,0}|$  by dropping the condition that $Z_t$
can only be visited at the end by any such path. Using~\eqref{sh} and $|Z_{t,i}|>|Z_t|f_t$ 
on the event $\mathcal{E}_t^{\ssup 1}$, we have 
\begin{align*}
|\mathcal{P}_{t,0,w,0}|
&\le \sum_{w_1+\cdots+w_d=w}{|Z_t|+2w\choose w_1, \dots, w_d, |Z_{t,1}|+w_1, \dots, |Z_{t,d}|+w_d} \\
&=N(Z_t)\frac{(|Z_t|+2w)!}{|Z_t|!}\sum_{w_1+\cdots+w_d=w}
\prod_{i=1}^d\frac{|Z_{t,i}|!}{(|Z_{t,i}|+w_i)!w_i!}\\
&\le N(Z_t) (|Z_t|+2w)^{2w}\sum_{w_1+\cdots+w_d=w}
\prod_{i=1}^d
\frac{1}{|Z_{t,i}|^{w_i}w_i!}\\
&\le N(Z_t) \frac{(|Z_t|+2w)^{2w}}{|Z_t|^w f_t^w w!}\sum_{w_1+\cdots+w_d=w}
{ w\choose w_1, \dots, w_d}
= N(Z_t) \frac{d^w(|Z_t|+2w)^{2w}}{|Z_t|^w f_t^w w!}. 
\end{align*}
If $2w\ge |Z_t|f_t$ then $\mathcal{P}_{t,0,w,0}=\emptyset$ by~\eqref{mwsR}. 
If $2w<|Z_t|f_t$
we obtain 
\begin{align*}
|\mathcal{P}_{t,0,w,0}|\le N(Z_t) \Big(\frac{2d|Z_t|}{f_t}\Big)^w\frac{1}{w!}.
\end{align*}
Estimating the contribution from the paths corresponding to $m$ and $s$ very roughly by the factor of $(2d)^s$ we obtain the required estimate. 
\end{proof}

The following lemma is an immediate consequence of the assumption $\alpha<2d$. It will be used repeatedly throughout the paper and plays a key role both in showing that only the shortest paths contribute to the asymptotics and in the disappearance of fluctuations in dimensions $d\ge2$.

\begin{lemma} 
\label{l_b2z}
Let  $\alpha\in (d, 2d)$. For any $c\in \R$, as $t\to\infty$
\begin{align*}
\beta_t^2 |Z_t|=o(f_t^c)
\end{align*}
almost surely on the event $\mathcal{E}_t^{\ssup 1}$.
\end{lemma}

\begin{proof} This follows immediately from $Y_t>f_t$ and $|X_t|<g_t$ on $\mathcal{E}_t^{\ssup 1}$
as well as~\eqref{ra} since
\begin{align*}
\beta_t^2 |Z_t|< \frac{r_tg_t}{a_t^2f_t^2}=\frac{g_t}{f_t^2}r_t^{1-2d/\alpha}=o(f_t^c)
\end{align*}
due to $\alpha<2d$. 
\end{proof}

The final lemma of this section shows that the penalties associated with paths from $\mathcal{P}_{t,m,w,s}$, which will arise later in the analysis, overpower the combinatorial growth of these spaces.

\begin{lemma} 
\label{l_npaths}
Let $\alpha\in (d, 2d)$.
For any $c\in \R$, as $t\to\infty$ 
\begin{align*}
\frac{1}{N(Z_t)}\sum_{(m,w,s)\neq 0}\frac{t^m}{(f_t^c \xi(Z_t))^{2w+s} m!}|\mathcal{P}_{t, m, w, s}|
=o(1)
\end{align*}
almost surely on the event $\mathcal{E}_t^{\ssup 1}$. 
\end{lemma}

\begin{proof} 
It follows from Lemma~\ref{l_paths} that 
\begin{align*}
\frac{1}{N(Z_t)}
\sum_{(m,w,s)\neq 0}
&\frac{t^m}{(f_t^c \xi(Z_t))^{2w+s} m!}|\mathcal{P}_{t, m, w, s}|
 \le 
 \sum_{w=0}^{\infty}\sum_{m=0}^{\infty}\sum_{s=m}^{\infty}
 \frac{1}{w!} \Big(\frac{2d|Z_t|}{f_t^{2c+1} \xi(Z_t)^2}\Big)^w\Big(\frac{2d}{f_t^c\xi(Z_t)}\Big)^s
  \frac{t^m}{m!} 
  -1\\
&=\exp\Big\{\frac{2d|Z_t|}{f_t^{2c+1} \xi(Z_t)^2}+\frac{2d t}{f_t^c\xi(Z_t)}\Big\}\Big(1-\frac{2d}{f_t^c\xi(Z_t)}\Big)^{-1}-1=o(1)
\end{align*}
by Lemma~\ref{l_b2z} and since $\frac{1}{f_t^c \xi(Z_t)}<\frac{t}{f_t^c \xi(Z_t)}<\frac{t}{a_t f_t^{c+1}}=o(1)$ 
as $t\to\infty$ on the event $\mathcal{E}_t^{\ssup 1}$.
\end{proof}
\bigskip


\section{One-dimensional case}

\label{s_1dim}

Throughout this section we assume that $d=1$ and $\alpha\in (1, 2)$, although we restate that in 
the main propositions. As there is only one shortest path between the origin and $Z_t$ we have 
$N(Z_t)=1$. 
\smallskip

In the following proposition, we show that $U(t)$ can be well approximated by the contribution of the unique shortest path from the origin to $Z_t$. The proof relies on Proposition~\ref{p_shortlong} and, in addition, exploits a specifically one-dimensional feature: every other path from the origin to $Z_t$ must pass through all the sites visited by the shortest path. This allows us to compare directly the contribution of an arbitrary path with that of the shortest path.

\begin{prop}
\label{p_u1dim}
Let $d=1$ and $\alpha\in (1,2)$. 
Almost surely on the event $\mathcal{E}_t^{\rm loc}\cap\mathcal{E}_t^{\ssup 1}\cap\mathcal{E}_t^{\ssup 2}$, as $t\to\infty$, 
\begin{align}
\label{pro1}
U(t)=e^{t\xi(Z_t)-|Z_t|\log\xi(Z_t)-2t}\prod_{j=0}^{|Z_t|-1}\Big(1-\frac{\xi(y_j^{\diamond})}{\xi(Z_t)}\Big)^{-1}+o\big(U(t)\big),
\end{align}
where $y^{\diamond}$ is the unique shortest path from the origin to $Z_t$. 
\end{prop}

\begin{proof} 
Let $y\in \mathcal{P}_{t, m, w, s}$. Since $d=1$, $y$ passes through  all points between zero and $Z_t$, allowing us to compare its contribution with that of $y^{\diamond}$. Using~\eqref{long} in Proposition~\ref{p_shortlong}, keeping only those visits and estimating the other terms in the product by $\xi(Z_t)-\xi(y_j)>f_t \xi(Z_t)$ on $\mathcal{E}_t^{\ssup 2}$, we obtain 
\begin{align*}
U(t,y)
&\le  e^{t\xi(Z_t)-2t}\frac{t^m}{m!} \frac{1}{(f_t \xi(Z_t))^{2w+s}}\prod_{j=0}^{|Z_t|-1}\frac{1}{\xi(Z_t)-\xi(y_j^{\diamond})}\\
&=e^{t\xi(Z_t)-|Z_t|\log \xi(Z_t)-2t}\frac{t^m}{m!} \frac{1}{(f_t \xi(Z_t))^{2w+s}}\prod_{j=0}^{|Z_t|-1}\Big(1-\frac{\xi(y_j^{\diamond})}{\xi(Z_t)}\Big)^{-1}.
\end{align*} 
Summing over all paths except the shortest one, we obtain using Lemma~\ref{l_npaths} on 
$\mathcal{E}_t^{\ssup 1}$
\begin{align*}
\sum_{y\in \mathcal{P}_t, y\neq y^{\diamond}} U(t, y) 
&\le e^{t\xi(Z_t)-|Z_t|\log \xi(Z_t)-2t}\prod_{j=0}^{|Z_t|-1}\Big(1-\frac{\xi(y_j^{\diamond})}{\xi(Z_t)}\Big)^{-1}
\sum_{(m,w,s)\neq 0}\sum_{y\in \mathcal{P}_{t,m,w,s}}\frac{t^m}{m!} \frac{1}{(f_t \xi(Z_t))^{2w+s}}\\
&\le e^{t\xi(Z_t)-|Z_t|\log \xi(Z_t)-2t}\prod_{j=0}^{|Z_t|-1}\Big(1-\frac{\xi(y_j^{\diamond})}{\xi(Z_t)}\Big)^{-1}
\sum_{(m,w,s)\neq 0}\frac{t^m}{m!} \frac{1}{(f_t \xi(Z_t))^{2w+s}}|\mathcal{P}_{t,m,w,s}|\\
&=e^{t\xi(Z_t)-|Z_t|\log \xi(Z_t)-2t}\prod_{j=0}^{|Z_t|-1}\Big(1-\frac{\xi(y_j^{\diamond})}{\xi(Z_t)}\Big)^{-1}o(1).
\end{align*}
On the other hand, it follows from~\eqref{short} in Proposition~\ref{p_shortlong} that on the event 
$\mathcal{E}_t^{\rm loc}\cap\mathcal{E}_t^{\ssup 1}\cap \mathcal{E}_t^{\ssup 2}$
\begin{align*}
U(t, y^{\diamond})=e^{t\xi(Z_t)-|Z_t|\log \xi(Z_t)-2t}\prod_{j=0}^{|Z_t|-1}\Big(1-\frac{\xi(y_j^{\diamond})}{\xi(Z_t)}\Big)^{-1}+o\big(U(t)\big), 
\end{align*}
and the claim follows from~\eqref{deco} and~\eqref{comploc0}. 
\end{proof}

Our next goal is to study the product terms appearing in~\eqref{pro1}. Observe that for all $j$
\begin{align*}
\xi(y_j^{\diamond})=\Psi_t(y_j^{\diamond})+q\frac{|y_j^{\diamond}|}{t}\log t<\Psi_t(Z_t)+q\frac{|Z_t|}{t}\log t=\xi(Z_t)
\end{align*}
since $|y_j^{\diamond}|<|Z_t|$ and since $Z_t$ is the maximiser of $\Psi_t$. Hence 
\begin{align*}
\prod_{j=0}^{|Z_t|-1}\Big(1-\frac{\xi(y_j^{\diamond})}{\xi(Z_t)}\Big)^{-1}
=\exp\Big\{\beta_t\sum_{j=0}^{|Z_t|-1}\zeta_t(y_j^{\diamond})\Big\},
\end{align*}
where $\zeta_t$ was defined in~\eqref{zeta}. It will, however, be more convenient to 
consider an approximately centred version
of this sum, namely, 
\begin{align*}
S_t=\beta_t\sum_{j=0}^{|Z_t|-1}\big[\zeta_t (y^{\diamond}_j) -\mu\big].
\end{align*}
Our main task is to understand the limiting behaviour of $S_t$ conditionally on $(X_t, Y_t)$, and 
to do so we introduce the conditional characteristic function of $S_t$ defined by 
\begin{align*}
\psi_t(s)=\Et \exp\{is S_t\}, \qquad s\in \R.
\end{align*}

As we want to describe the conditional limit of $S_t$ in terms of a time-inhomogeneous Poisson process, 
we define, for any $a, b>0$  such that $b-q a>0$, its intensity measure 
\begin{align*}
L_{a, b}(dx\otimes dv)=\frac{\alpha e^{-v}}{b^{\alpha}(1-e^{-v})^{\alpha+1}}
dx\otimes dv 
\Big|_{D_{a,b}},
\end{align*}
where 
\begin{align*}
D_{a,b}=\Big\{(x,v)\in (0,a)\times (0,\infty): v<\log\frac{b}{q(a-x)}\Big\}.
\end{align*}

\begin{lemma} 
\label{l_levy}
$L_{a,b}$ is indeed an intensity measure, and therefore
\begin{align*}
\phi_{a,b}(s)=\exp\Big\{\int_{D_{a,b}}(e^{isv}-1-isv)L_{a,b}(dx\otimes dv)\Big\}
\end{align*}
is the characteristic function of a random variable $W_{a,b}$ that is the value at time $a$ of a time-inhomogeneous compensated Poisson process with 
zero drift, no Brownian component and the intensity measure $L_{a,b}$. 
\end{lemma}

\begin{proof} 
Substituting $v=-\log(1-u)$ we obtain 
\begin{align*}
\int_{D_{a,b}}\min\{1, v^2\}L_{a,b}(dx\otimes dv)
&=\frac{1}{b^{\alpha}}\int_0^a \int_0^{1-\frac{q(a-x)}{b}}\!\!\!\!\!\!\!\!
\min\{1, \log^2(1-u)\}
\frac{\alpha}{u^{\alpha+1}}du dx\\
&\le \frac{a}{b^{\alpha}} \int_0^1
\min\{1, \log^2(1-u)\}
\frac{\alpha}{u^{\alpha+1}}du<\infty
\end{align*}
since $\alpha<2$. Since the integrability condition is satisfied,  $L_{a,b}$ is indeed an intensity measure, which also implies the rest of the 
statement.
\end{proof}

In addition to the time-inhomogeneous compensated Poisson process above, we also need the shift term 
\begin{align*}
\gamma_{a,b}= \int_0^{a}\int_{0}^{\infty}\frac{\alpha}{y^{\alpha+1}}
\Big[\frac{y}{b}+\log\Big(1-\frac{y}{b}\Big)\1_{\{y-q x<b-q a\}}\Big]dydx
\end{align*}
to describe the limiting behaviour of $S_t$. 
With all the ingredients now in place, we show in the proposition below that, conditionally on 
$(X_t, Y_t)$, the distribution of $S_t$ is approximately that of the shifted 
 time-inhomogeneous compensated Poisson process $W_{|X_t|, Y_t}-\gamma_{|X_t|, Y_t}$.
\smallskip

We work on the typical event
\begin{align*}
\mathcal{E}_{t,\delta}=\Big\{Y_t>\delta, 1-q\frac{|X_t|}{Y_t}>\delta, 
\frac{Y_t^{\alpha}}{|X_t|}>\delta\Big\}, \qquad \delta>0,
\end{align*}
which causes no difficulty, since Proposition~\ref{events} shows that its probability can be made arbitrarily close to $1$ by choosing $\delta$ sufficiently small. Unlike the other typical events used throughout the paper, however, here it is important to work with a fixed boundary $\delta$, rather than with the moving boundaries $f_t$ and $g_t$.

\begin{prop} 
\label{p_psi}
For any $\delta>0$ and any $s\in \R$ 
\begin{align*}
\psi_t(s)= \phi_{|X_t|, Y_t}(s)\exp\{-is\gamma_{|X_t|, Y_t}+o(1)\}
\end{align*}
as $t\to\infty$ almost surely on $\mathcal{E}_{t, \delta}$.
\end{prop}

\begin{proof} It is easy to see that, 
conditionally on $\mathcal{G}_t$, the random variables $\{\xi(z): z\neq Z_t\}$ are 
independent and each $\xi(z)$ has Pareto distribution with parameter $\alpha$ conditioned on 
$\Psi_t(z)<\Psi_t(Z_t)$, which is equivalent to $\xi(z)<\xi(Z_t)-\frac{|Z_t|-|z|}{t}q\log t$. 
For each $t$ and $z$, denote
\begin{align}
\label{rho0}
\rho_{t,z}=\beta_t\Big[\xi(Z_t)-\frac{|Z_t|-|z|}{t}q\log t\Big].
\end{align}
Observe that according to~\eqref{ra}
\begin{align*}
\rho_{t,z}=1-\frac{q}{Y_t}\Big(|X_t|-\frac{|z|}{r_t}\Big),
\end{align*}
which on the event $\mathcal{E}_{t,\delta}$  implies 
for each $z$ such that $ |z|< |Z_t|$ and $z Z_t\ge 0$
\begin{align}
\label{d5}
1> \rho_{t,z}\ge 1-q\frac{|X_t|}{Y_t}>\delta. 
\end{align}

First, we compute the normalising factor of the conditioning to be 
\begin{align}
\label{p1}
\int_1^{\infty}\frac{\alpha  dy}{y^{\alpha+1}}\1_{\{y-\frac{|z|}{t}q\log t<\Psi_t(Z_t)\}}
= \int_1^{\frac{\rho_{t,z}}{\beta_t}}\frac{\alpha dy}{y^{\alpha+1}}=1-\beta_t^{\alpha}
\rho_{t,z}^{-\alpha}.
\end{align}

Further, using~\eqref{p1} we get for any $s\in \R$
\begin{align}
\Et \exp\big\{is\beta_t\zeta_t(z)\big\}
&=\big(1-\beta_t^{\alpha}\rho_{t,z}^{-\alpha}\big)^{-1}
\int_1^{\frac{\rho_{t,z}}{\beta_t}}\frac{\alpha}{y^{\alpha+1}}
\exp\big\{-is\log\big(1-\beta_t y\big)\big\}dy\notag\\
&=\big(1-\beta_t^{\alpha}\rho_{t,z}^{-\alpha}\big)^{-1}\beta_t^{\alpha}
\int_{\beta_t}^{\rho_{t,z}}\frac{\alpha}{y^{\alpha+1}}(1-y)^{-is}dy. 
\label{p3}
\end{align}

We will use the Taylor expansion for $(1-y)^{-is}$ with respect to $y$. 
For the first two terms we have 
\begin{align}
\label{in}
\beta_t^{\alpha}\int_{\beta_t}^{\rho_{t,z}}\frac{\alpha dy }{y^{\alpha+1}}
=1-\beta_t^{\alpha}\rho_{t,z}^{-\alpha}
\qquad\text{ and }\qquad
\beta_t^{\alpha}\int_{\beta_t}^{\rho_{t,z}}\frac{\alpha dy}{y^{\alpha}}
=\mu \beta_t-\mu \beta_t^{\alpha}\rho_{t,z}^{1-\alpha}.
\end{align}

For the remainder  term we have uniformly in $z$
\begin{align}
\label{rem}
\beta_t^{\alpha}\int_{\beta_t}^{\rho_{t,z}}\frac{\alpha }{y^{\alpha+1}}\Big[(1-y)^{-is}-1-isy\Big]dy
=\beta_t^{\alpha}\int_{0}^{\rho_{t,z}}\frac{\alpha }{y^{\alpha+1}}\Big[(1-y)^{-is}-1-isy\Big]dy
+o(\beta_t^{\alpha}).
\end{align}

Substituting~\eqref{in} and~\eqref{rem} into~\eqref{p3} we get uniformly in $z$ on the event $\mathcal{E}_{t,\delta}$,
using uniform boundedness~\eqref{d5} of $\rho_{t,z}$ and $\rho_{t,z}^{-1}$, 
as well as $\beta_t^2=o(\beta_t^{\alpha})$ according to $\alpha<2$, 
that
\begin{align*}
\Et & \exp\big\{is\beta_t\zeta_t(z)\big\}
=\big(1-\beta_t^{\alpha}\rho_{t,z}^{-\alpha}\big)^{-1}
\Big(1-\beta_t^{\alpha}\rho_{t,z}^{-\alpha}
+is\mu \beta_t-is\mu \beta_t^{\alpha}\rho_{t,z}^{1-\alpha}
\\&\phantom{aaaaaaaaaaaaaaaaa}
+\beta_t^{\alpha}\int_{0}^{\rho_{t,z}}\frac{\alpha }{y^{\alpha+1}}\Big[(1-y)^{-is}-1-isy\Big]dy
+o(\beta_t^{\alpha})\Big)\\
&=1
+is\mu \beta_t-is\mu \beta_t^{\alpha}\rho_{t,z}^{1-\alpha}
+\beta_t^{\alpha}\int_{0}^{\rho_{t,z}}\frac{\alpha }{y^{\alpha+1}}\Big[(1-y)^{-is}-1-isy\Big]dy
+o(\beta_t^{\alpha})\\
&=\exp\Big\{is\mu \beta_t-is\mu \beta_t^{\alpha}\rho_{t,z}^{1-\alpha}
+\beta_t^{\alpha}\int_{0}^{\rho_{t,z}}\frac{\alpha }{y^{\alpha+1}}\Big[(1-y)^{-is}-1-isy\Big]dy
+o(\beta_t^{\alpha})\Big\}.
\end{align*}

Since $\beta_t^{\alpha} |Z_t|=|X_t|Y_t^{-\alpha}$
is uniformly bounded on $\mathcal{E}_{t,\delta}$ and due to $\rho_{t,z}=\rho_{t,-z}$, 
this yields 
\begin{align}
\psi_t(s)
&=\prod_{j=0}^{|Z_t|-1}
\Big[\exp\big\{-is\mu\beta_t\big\}\Et  \exp\big\{is\beta_t\zeta_t(y^{\diamond}_j)\big\}\Big]\notag\\
&=\exp\Big\{-is\mu \beta_t^{\alpha}\sum_{z=0}^{|Z_t|-1}\rho_{t,z}^{1-\alpha}
+\beta_t^{\alpha}\sum_{z=0}^{|Z_t|-1}\int_{0}^{\rho_{t,z}}\frac{\alpha }{y^{\alpha+1}}\Big[(1-y)^{-is}-1-isy\Big]dy
+o(1)\Big\}.
\label{d6}
\end{align}

Observe that on the event $\mathcal{E}_{t,\delta}$ the function $x\mapsto 1-\frac{q}{Y_t}(|X_t|-x)$ has a uniformly bounded derivative. Hence we can pass to the integral limits in~\eqref{d6}
using $\beta_t^{\alpha}=Y_t^{-\alpha} \frac{1}{r_t}$, 
$Y_t$ being bounded away from zero, and~\eqref{d5} on the event $\mathcal{E}_{t,\delta}$. We  obtain 

\begin{align}
\psi_t(s)
&=\exp\Big\{-\frac{is\mu}{Y_t^{\alpha}} \cdot \frac{1}{r_t}\sum_{z=0}^{|Z_t|-1}\rho_{t,z}^{1-\alpha}
+\frac{1}{Y_t^{\alpha}} \cdot \frac{1}{r_t}\sum_{z=0}^{|Z_t|-1}\int_{0}^{\rho_{t,z}}\frac{\alpha }{y^{\alpha+1}}\Big[(1-y)^{-is}-1-isy\Big]dy
+o(1)\Big\}\notag\\
&=\exp\Big\{-\frac{is\mu}{Y_t^{\alpha}} \int_{0}^{|X_t|} 
\Big[1-\frac{q}{Y_t}(|X_t|-x)\Big]^{1-\alpha}dx\notag\\
&\phantom{aaaaaaa}+\frac{1}{Y_t^{\alpha}} \int_0^{|X_t|}\int_{0}^{1-\frac{q}{Y_t}(|X_t|-x)}\frac{\alpha }{y^{\alpha+1}}\Big[(1-y)^{-is}-1-isy\Big]dydx
+o(1)\Big\}.
\label{d7}
\end{align}

For the first term, we have 
\begin{align}
\mu \int_{0}^{|X_t|} \Big[1-\frac{q}{Y_t}(|X_t|-x)\Big]^{1-\alpha}dx
=\int_{0}^{|X_t|}\int_{1-\frac{q}{Y_t}(|X_t|-x)}^{\infty} 
\frac{\alpha y}{y^{\alpha+1}}dydx.
\label{d8}
\end{align}

In order to proceed with the change of variables $v=-\log(1-y)$ in the second term of~\eqref{d7}, we 
would like to replace the linear term $y$ by $-\log(1-y)$, which performs similar compensation around zero but leads to an extra drift term away from zero. Together with~\eqref{d8}
this leads to 

\begin{align*}
\psi_t(s)
&=\exp\Big\{-\frac{is}{Y_t^{\alpha}} \int_{0}^{|X_t|}\int_{0}^{\infty} 
\frac{\alpha}{y^{\alpha+1}}\Big[y+\log(1-y)\1_{\big\{y<1-\frac{q(|X_t|-x)}{Y_t}\big\}}\Big]dydx\notag\\
&\phantom{aaaaaaa}+\frac{1}{Y_t^{\alpha}} \int_0^{|X_t|}\int_{0}^{1-\frac{q}{Y_t}(|X_t|-x)}\frac{\alpha }{y^{\alpha+1}}\Big[(1-y)^{-is}-1+is\log(1-y)\Big]dydx
+o(1)\Big\}.
\end{align*}

Rescaling $y$ in the first term and changing $y$ to $v$ in the second term, we obtain 
\begin{align*}
\psi_t(s)
&=\exp\Big\{-is\gamma_{|X_t|, Y_t}
+ \int_{D_{|X_t|, Y_t}}(e^{isv}-1-isv)L_{|X_t|, Y_t}(dx\otimes dv)
+o(1)\Big\},
\end{align*}
which completes the proof. 
\end{proof}


\begin{proof}[Proof of Theorem~\ref{th_1}] It follows from Proposition~\ref{p_u1dim}
that on the event $\mathcal{E}_t^{\rm loc}\cap\mathcal{E}_t^{\ssup 1}\cap \mathcal{E}_t^{\ssup 2}$
\begin{align*}
U(t)=\exp\Big\{t\xi(Z_t)-|Z_t|\log \xi(Z_t)-2t+\mu\frac{|Z_t|}{\xi(Z_t)}
+S_t\Big\}+o\big(U(t)\big).
\end{align*}
Further,  on the event $\{Z_t=Z_t^*\}\cap \mathcal{E}_t^{\ssup 1}$ we have 
\begin{align*}
t\xi(Z_t)-|Z_t|\log \xi(Z_t)=t\Psi_t^*(Z_t^*). 
\end{align*}
This implies~\eqref{asymp} with $\chi_t=S_t+o(1)$ on $\mathcal{E}_t^{\rm loc}\cap\mathcal{E}_t^{\ssup 1}\cap \mathcal{E}_t^{\ssup 2}\cap \{Z_t=Z_t^*\}$. 
\smallskip

We have  $\text{\rm P}(\mathcal{E}_t^{\rm loc})\to 1$ by~\eqref{comploc0}, 
$\text{\rm P}(Z_t=Z_t^*)\to 1$ by Proposition~\ref{zz} and
$\text{\rm P}(\mathcal{E}_t^{\ssup 1}\cap \mathcal{E}_t^{\ssup 2})\to 1$ by Proposition~\ref{events}. 
Also by Proposition~\ref{zz} we know that $\Psi_t^*(Z_t^*)=(1+o(1))\Psi_t(Z_t)$
with probability converging to $1$. 
Hence it remains to 
show that $(X_t, \Upsilon_t, Y_t, S_t)$ converges in distribution to a random variable
$(X, \Upsilon, \Upsilon + q|X|,\chi)$, 
and that $(X, \Upsilon, \chi)$ has a positive density on $\R\times (0,\infty)\times \R$.
\smallskip 

Due to the deterministic relation $Y_t=\Upsilon_t+q|X_t|$ following from~\eqref{ra}, it 
will be more convenient to denote $Y=\Upsilon +q|X|$ and prove an equivalent statement, namely, that 
$(X_t, Y_t, S_t)$ converges in distribution to 
$(X,  Y,\chi)$, 
and that $(X, Y, \chi)$ has a positive density on $\{(x,y): x\in \R, y>q|x|\}\times \R$. 
We denote the  expectation on the probability space of the limit random variable by $\Ep$.
\smallskip 

Consider the characteristic function of $(X_t, Y_t, S_t)$
\begin{align*}
\phi_t(x, y, s)=\text{\rm E} \exp\{i(x X_t+y Y_t+s S_t)\}
=\text{\rm E}\big[\exp\{i(x X_t+y Y_t)\}\psi_t(s)\big].
\end{align*}
Let $\e>0$ and choose $\delta$ according to Proposition~\ref{events} that ensures 
${\text P}(\mathcal{E}_{t,\delta})>1-\e$ eventually. By Proposition~\ref{p_psi} we have 
\begin{align}
\label{o12}
\Big|\phi_t(x, y, s)
-\text{\rm E}\big[\exp\big\{i(x X_t+y Y_t)-is\gamma_{| X_t|, Y_t}\big\}
\phi_{|X_t|, Y_t}(s)\big]\Big|\le 2\e+o(1).
\end{align}

Let $\chi$ be defined 
as $W_{|X|, Y}-\gamma_{|X|, Y}$ from Lemma~\ref{l_levy}.   
Using  $(X_t, Y_t)\Rightarrow (X,Y)$ by Proposition~\ref{p_new},   Lemma~\ref{l_levy} and 
boundedness and continuity of the 
function under the expectation, we obtain 
\begin{align*}
\lim_{t\to\infty}\text{\rm E}\big[\exp\big\{i(x X_t+y Y_t)-is\gamma_{|X_t|,  Y_t}\big\}
\phi_{|X_t|, Y_t}(s)\big]
&=\Ep\big[\exp\big\{i(x X+y Y)-is\gamma_{|X|, Y}\big\}
\phi_{|X|, Y}(s)\big]\\
&=\Ep \exp\{i(x X+y Y+s \chi)\}.
\end{align*}
Combining this with~\eqref{o12} we get
\begin{align*}
\lim_{t\to\infty}\phi_t(x, y, s)=\Ep \exp\{i(x X+y Y+s \chi)\}
\end{align*}
since $\e$ is arbitrary, thus implying convergence in distribution of $(X_t, Y_t, S_t)$ to $(X, Y, \chi)$. 
\smallskip

Finally, $(X, Y)$ 
has a positive density on $\{(x,y): x\in \R, y>q|x|\}$ by Proposition~\ref{p_new} and,  conditionally on $(X, Y)$, 
$\chi$ has positive density on $\R$ as a shifted time-inhomogeneous compensated Poisson process
with intensity measure $L_{|X|,Y}$. This implies that the triple $(X, Y, \chi)$ has positive density on 
the domain $\{(x,y): x\in \R, y>q|x|\}\times \R$. 
\end{proof}

\bigskip


\section{Preliminary estimates for $d\ge 2$}
\label{s_prelim}

We now turn to the multidimensional case, where the analysis is based on the centred and normalised polymer weights $H_t$ defined in~\eqref{ht}. We begin by deriving the asymptotic behaviour of the corresponding centring and normalisation terms.

\begin{lemma} 
\label{l_ml}
Let $d\ge 2$. Almost surely on the event $\mathcal{E}_t^{\ssup 1}$, as $t\to\infty$
\begin{align}
\label{mu}
\mu_{t}(z)&=\mu+O(\beta_t),\\
\lambda_{t,z}(c\beta_t) & = O(\beta_t^2)
\label{lambda}
\end{align}
uniformly for all $z\neq Z_t$ and $c\in \{1, 2\}$. 
\end{lemma}

\begin{proof} 
Similarly to the proof of Proposition~\ref{p_psi},
conditionally on $\mathcal{G}_t$,  $\{\xi(z): z\neq Z_t\}$ are 
independent and each such $\xi(z)$ has Pareto distribution with parameter $\alpha$ conditioned on 
$\Psi_t(z)<\Psi_t(Z_t)$. 
\smallskip

Observe that for $z\neq Z_t$ the normalising factor of the conditioning satisfies
\begin{align*}
1
&\ge  
\int_1^{\infty}\frac{\alpha  dy}{y^{\alpha+1}}\1_{\{y-\frac{|z|}{t}q\log t<\Psi_t(Z_t)\}}
\ge \int_1^{\Psi_t(Z_t)}\frac{\alpha dy}{y^{\alpha+1}}=1-\Psi_t(Z_t)^{-\alpha}.
\end{align*}
On the event $\mathcal{E}_t^{\ssup 1}$ we have
\begin{align*}
\beta_t \Psi_t(Z_t)=1-q\frac{|X_t|}{Y_t}>f_t
\end{align*}
which implies, since $\alpha>d\ge 2$, that 
\begin{align}
\label{1}
\int_1^{\infty}\frac{\alpha  dy}{y^{\alpha+1}}\1_{\{y-\frac{|z|}{t}q\log t<\Psi_t(Z_t)\}}=1+o(\beta_t^2).
\end{align}
\smallskip

Let us now show~\eqref{mu}. Using~\eqref{1} we have for all $z\neq Z_t$ 
\begin{align}
\mu_t(z)
&=\big(1+o(\beta_t^2)\big)
\int_1^{\xi(Z_t)(1-f_t)}\frac{\alpha}{y^{\alpha+1}}\Big[-\xi(Z_t)\log\Big(1-\frac{y}{\xi(Z_t)}\Big)\Big]
\1_{\{y-\frac{|z|}{t}q\log t<\Psi_t(Z_t)\}}dy\notag\\
&=\big(1+o(\beta_t^2)\big)\beta_t^{\alpha-1}\Big[B(\beta_t)-B\big(\min\{1-f_t, \rho_{t,z}\}\big)\Big],
\label{t1}
\end{align}
where $\rho_{t,z}$ is defined by~\eqref{rho0} and 
\begin{align}
\label{b}
B(x)=\int_x^1 \frac{\alpha}{s^{\alpha+1}}\big[-\log(1-s)\big]ds = \mu x^{1-\alpha}+O(x^{2-\alpha})
\qquad \text{as }x\downarrow 0.
\end{align}
Observe that on $\mathcal{E}_t^{\ssup 1}$ 
\begin{align}
\label{rho}
\rho_{t,z}\ge 1-q\frac{|X_t|}{Y_t}>f_t.
\end{align}
%
This together with~\eqref{b} implies that the subtracted term in~\eqref{t1} is negligible, 
and the asymptotics of the first term in~\eqref{t1} yields~\eqref{mu}.
\smallskip

Finally, let us show~\eqref{lambda}. Using~\eqref{1} we have   
\begin{align}
\Et \exp\{c\beta_t\zeta_t(z)\}
&=\big(1+o(\beta_t^2)\big)
\int_1^{\xi(Z_t)(1-f_t)}\frac{\alpha}{y^{\alpha+1}}\Big(1-\frac{y}{\xi(Z_t)}\Big)^{-c}\1_{\{y-\frac{|z|}{t}
q\log t<\Psi_t(Z_t)\}}dy\notag\\
&=\big(1+o(\beta_t^2)\big)\beta_t^{\alpha}
\Big[\hat B_c(\beta_t)-\hat B_c\big(\min\{1-f_t, \rho_{t,z}\}\big)\Big],
\label{t2}
\end{align}
where 
\begin{align}
\label{bhat}
\hat B_c(x)&=\int_x^{\frac 1 2}\frac{\alpha ds}{s^{\alpha+1}(1-s)^c}=
\left\{
\begin{array}{ll}
x^{-\alpha}+c\mu x^{1-\alpha}+O(x^{2-\alpha}) &\qquad \text{as }x\downarrow 0,\\
O\big((1-x)^{-1}\big) & \qquad \text{as }x\uparrow 1.
\end{array}\right.
\end{align}
It follows from~\eqref{rho} and~\eqref{bhat} that the subtracted term 
in~\eqref{t2} is negligible on $\mathcal{E}_t^{\ssup 1}$, and the asymptotics given by the first term yields
\begin{align*}
\Et \exp\{c\beta_t\zeta_t(z)\}=1+c\mu \beta_t+O(\beta_t^2). 
\end{align*}
This, together with~\eqref{mu}, implies 
\begin{align*}
\lambda_{t,z}(c\beta_t)=\log \Et \exp\{c\beta_t\zeta_t(z)\}-c\beta_t \mu_t(z)=O(\beta_t^2)
\end{align*}
as required. 
\end{proof}

To handle path intersections later in the paper, we introduce the nonlinearity correction associated with 
$\lambda_{t,z}$. Namely, for each $t$ and $z$ we define
\begin{align}
\label{lat}
\Lambda_{t,z}(c)=\lambda_{t,z}(c\beta_t)-c\lambda_{t,z}(\beta_t), \qquad c\in \R.
\end{align}
In the lemma below, the case $c=2$ corresponds to the typical situation where two paths intersect at a site that is visited once by each path. Since such intersections are common, we need a relatively precise asymptotic estimate in this case. Values $c\ge3$ arise from rarer intersection patterns, involving multiple visits to the same site, and for these an upper bound is sufficient.

\begin{lemma}
\label{l_lll}
Let $d\ge 2$.  Almost surely on the event $\mathcal{E}_t^{\ssup 1}$, 
the following holds. 
\begin{itemize}
\item[(1)] As $t\to\infty$, uniformly for all $z\neq Z_t$
\begin{align*}
\Lambda_{t,z}(2)&=O(\beta_t^2)
\end{align*}

\item[(2)] For all $t$ and $z\neq Z_t$
\begin{align*}
|\Lambda_{t,z}(c)|&\le c \log(1/f_t) \qquad\text{ for all }c\ge 0.
\end{align*}
\end{itemize}
\end{lemma}

\begin{proof} The first statement follows immediately from~\eqref{lambda} in Lemma~\ref{l_ml}. 
\smallskip

To prove the second statement, observe that for any $c\ge 0$ we have 
\begin{align}
\label{in1}
\lambda_{t,z}(c\beta_t)
&=\log \Et \exp\Big\{-c\log \Big(1-\frac{\xi(z)}{\xi(Z_t)}\Big)\1_{\{\xi(z)<\xi(Z_t)(1-f_t)\}}\Big\} -c\beta_t\mu_t(z)
\le c\log(1/f_t)
\end{align}
by dropping the positive subtracted term and by using the indicator function condition to estimate the first term. On the other hand, by Jensen's inequality 
\begin{align}
\label{in2}
\lambda_{t,z}(c\beta_t)\ge \Et \Big\{c\beta_t\big[\zeta_t(z)-\mu_t(z)\big]\Big\}=0. 
\end{align}
Now the inequalities~\eqref{in1} and~\eqref{in2}, used with an arbitrary $c$ and $c=1$, give the second statement. 
\end{proof}
\bigskip


\section{Multidimensional case}
\label{s_multi}

As in the one-dimensional case, our starting point is Proposition~\ref{p_shortlong}. In higher dimensions, however, the product terms along a given path can no longer be compared directly with those along a shortest path, since a general path need not contain a distinguished shortest path as a subset. We therefore take a different approach.
\smallskip

In Proposition~\ref{p_product} below, we rewrite the product terms via the centred and normalised polymer weights $H_t$, using the asymptotic information on the centring and normalisation parameters obtained in Lemma~\ref{l_ml}.

\begin{prop} 
\label{p_product}
Let $d\ge 2$ and $\alpha\in (d, 2d)$. 
Almost surely on the event $\mathcal{E}_t^{\ssup 1}\cap \mathcal{E}_t^{\ssup 2}$, 
as $t\to\infty$,
\begin{align}
\label{t4}
\prod_{\substack{i=0\\ y_i\neq Z_t}}^{\ell(y)}\frac{1}{\xi(Z_t)-\xi(y_i)}
&=\exp\Big\{-|Z_t|\log \xi(Z_t)+\mu\frac{|Z_t|}{\xi(Z_t)}+o(1)\Big\} 
\Big[\frac{1 +o(1)}{\xi(Z_t)}\Big]^{2w+s}(1+H_t(y))
\end{align}
uniformly for all $(m,w,s)$ and all $y\in\mathcal{P}_{t,m,w,s}$.
\end{prop}

\begin{proof} For any $y\in\mathcal{P}_{t,m,w,s}$ we have $|y_i|<R_t$ for all $i$.  
Hence on the event $\mathcal{E}_t^{\ssup 2}$ 
we have 
\begin{align}
\label{t3}
\prod_{\substack{i=0\\ y_i\neq Z_t}}^{\ell(y)}\frac{1}{\xi(Z_t)-\xi(y_i)}
&=\frac{1}{\xi(Z_t)^{\ell(y)-m}}\prod_{\substack{i=0\\ y_i\neq Z_t}}^{\ell(y)} 
\exp\Big\{\frac{\zeta_t(y_i)}{\xi(Z_t)}\Big\}.
\end{align}
Further, Lemma~\ref{l_ml} on $\mathcal{E}_t^{\ssup 1}$ implies that uniformly
\begin{align*}
\frac{\zeta_t(y_i)}{\xi(Z_t)}
=\beta_t\big[\zeta_t(y_i)-\mu_t(y_i)\big]-\lambda_{t,y_i}(\beta_t)+\mu \beta_t+O(\beta_t^2).
\end{align*}
Substituting this into~\eqref{t3}, using $\ell(y)=|Z_t|+m+2w+s$, the number of terms in the product being 
$|Z_t|+2w+s<R_t$,  $\beta_t=o(1)$ and Lemma~\ref{l_b2z} on the event $\mathcal{E}_t^{\ssup 1}$ we obtain 
\begin{align*}
\prod_{\substack{i=0\\ y_i\neq Z_t}}^{\ell(y)}\frac{1}{\xi(Z_t)-\xi(y_i)}
&=\exp\Big\{-|Z_t|\log \xi(Z_t)+\mu \frac{|Z_t|}{\xi(Z_t)}+o(1)\Big\}\Big[\frac{1+o(1)}{\xi(Z_t)}\Big]^{2w+s}\\&\times \prod_{\substack{i=0\\ y_i\neq Z_t}}^{\ell(y)} 
\exp\Big\{\beta_t\big[\zeta_t(y_i)-\mu_t(y_i)\big]-\lambda_{t,y_i}(\beta_t)\Big\}, 
\end{align*}
which is equivalent to~\eqref{t4}.
\end{proof}

Denote
\begin{align*}
A_t=t\Big[\xi(Z_t)-\frac{|Z_t|}{t}\log \xi(Z_t)+\frac{\log N(Z_t)}{t}\Big]-2dt+\mu\frac{|Z_t|}{\xi(Z_t)}
\end{align*}
and observe that it is equal to the desired asymptotic~\eqref{asymp_d}  
\begin{align}
\label{at}
A_t=t\Psi_t^*(Z_t^*)-2dt+\mu\frac{|Z_t^*|}{\xi(Z_t^*)}
\end{align}
on the event $\mathcal{E}_t^{\ssup 1}\cap\{Z_t=Z_t^*\}$. Here the event $\mathcal{E}_t^{\ssup 1}$
is required to satisfy the indicator function condition in the definition of $\Psi_t^*$. 
\smallskip

In the proposition below, we rewrite the total contribution of all paths so that the 
asymptotic~\eqref{asymp_d}   is made explicit, and separate the remainder into two error terms: one arising from the shortest paths and the other from the contribution of all non-shortest paths.

\begin{prop} 
\label{p_xixi}
Let $d\ge 2$ and $\alpha\in (d,2d)$. 
Almost surely on the event 
$\mathcal{E}_t^{\rm loc}\cap \mathcal{E}_t^{\ssup 1}\cap \mathcal{E}_t^{\ssup 2}\cap\{Z_t=Z_t^*\}$
\begin{align}
\label{uu5}
U_0(t)=\exp\Big\{t\Psi_t^*(Z_t^*)-2dt+\mu\frac{|Z_t^*|}{\xi(Z_t^*)}+o(1)\Big\}\big(1+\Xi_t^{\rm min}+\Xi_t\big)+o(U(t)), 
\end{align}
as $t\to\infty$, where 
\begin{align*}
\Xi_t^{\rm min}=\frac{1}{N(Z_t)}\sum_{y\in \PP}H_t(y)
\qquad\text{and}\qquad 
\Xi_t=e^{-A_t}\sum_{y\in \mathcal{P}_t\setminus\PP}U(t,y).
\end{align*}
\end{prop}

\begin{proof}  
Observe that  
\begin{align*}
U_0(t)
&=\sum_{y\in \PP}U(t,y)+
\sum_{y\in \mathcal{P}_{t}\setminus \PP} U(t,y). 
\end{align*}

For the first term, we use~\eqref{short} in Proposition~\ref{p_shortlong} 
 followed by 
Proposition~\ref{p_product} to get  
\begin{align}
U_0(t)
&=e^{t\xi(Z_t)-2dt}\sum_{y\in \PP}
\prod_{j=0}^{|Z_t|-1}\frac{1}{\xi(Z_t)-\xi(y_j)} 
+e^{A_t}\Xi_t
+ o\big(U(t)\big)\notag\\
&=e^{A_t+o(1)}\frac{1}{N(Z_t)}\sum_{y\in \PP}
(1+H_t(y))
+e^{A_t}\Xi_t
+ o\big(U(t)\big)
\label{e4}
\end{align}
which yields~\eqref{uu5} due to~\eqref{at}.
\end{proof}

While the first error term $\Xi_t^{\rm min}$ is already expressed via the polymer weights $H_t$, the second error term $\Xi_t$ is not. This is because Proposition~\ref{p_shortlong} provides an asymptotic expression for the  shortest paths, but only an upper bound for the  non-shortest paths. At this stage, however, our goal is simply to show that both error terms are negligible, so this upper bound will now be used in the following proposition to estimate $\Xi_t$ in terms of $H_t$.
\smallskip

It is also useful to keep in mind the following. In the calculation of the contribution of the shortest paths in~\eqref{e4}, each path $y$ contributed a factor $1+H_t(y)$. Summing the deterministic parts $1$ over all shortest paths produced the factor $N(Z_t^*)$ in the asymptotic expression~\eqref{asymp_d}, while the fluctuation terms $H_t(y)$ gave rise to $\Xi_t^{\rm min}$. For non-shortest paths in the proposition below, each path will again contribute a factor $1+H_t(y)$. In this case, however, the terms involving $H_t(y)$ will form the main part of the upper bound, whereas the deterministic contributions $1$ will turn out to be negligible, since the penalties incurred by non-shortest paths dominate the combinatorial growth of the corresponding path spaces.

\begin{prop}
\label{p_xi}
Let $d\ge 2$ and $\alpha\in (d,2d)$. 
Almost surely on the event 
$\mathcal{E}_t^{\rm loc}\cap \mathcal{E}_t^{\ssup 1}\cap \mathcal{E}_t^{\ssup 2}$, as $t\to \infty$
\begin{align}
\label{upbo}
\Xi_t\le  \frac{2} {N(Z_t)}\sum_{(m,w,s)\neq 0} \sum_{y\in \mathcal{P}_{t, m, w, s}}
  \frac{t^m}{ \xi(Z_t)^{2w+s}m!}H_t(y) +o(1).
\end{align}
\end{prop}

\begin{proof} Using~\eqref{long} in Proposition~\ref{p_shortlong}  followed by 
Proposition~\ref{p_product} we have for any $y\in \mathcal{P}_{t,m,w,s}$
\begin{align*}
U(t, y)\le \frac{2e^{A_t}}{N(Z_t)} \frac{t^m}{ \xi(Z_t)^{2w+s}m!}(1+H_t(y)).
\end{align*}
Summing over all paths in $\mathcal{P}_t\setminus \PP$ and using Lemma~\ref{l_paths} 
on the event $\mathcal{E}_t^{\ssup 1}$ we obtain 
\begin{align*}
\Xi_t
&\le \frac{2}{N(Z_t)}\sum_{(m,w,s)\neq 0}\sum_{y\in \mathcal{P}_{t, m, w, s}}
  \frac{t^m}{ \xi(Z_t)^{2w+s}m!}(1+H_t(y))\notag\\
&=  \frac{2}{N(Z_t)}\sum_{(m,w,s)\neq 0}\Big[
  \frac{t^m}{ \xi(Z_t)^{2w+s}m!}|\mathcal{P}_{t, m, w, s}|+\sum_{y\in \mathcal{P}_{t, m, w, s}}
  \frac{t^m}{ \xi(Z_t)^{2w+s}m!}H_t(y)\Big].
\end{align*}
It remains to observe that by Lemma~\ref{l_npaths} on the event $\mathcal{E}_t^{\ssup 1}$ the second 
term provides the required bound and the first term 
contributes $o(1)$. 
\end{proof}

Our aim now is to show that the error terms $\Xi_t^{\rm min}$ and $\Xi_t$ are negligible. We will use $L^2$-estimates, so the correlations between the polymer weights $H_t$ become important. These correlations are determined by the extent to which the underlying paths intersect. In the proposition below, we show that the correlations between shortest paths are small. This is a consequence of the assumption $\alpha<2d$, which makes the inverse temperature $\beta_t$ sufficiently small to suppress the effect of path intersections. We also show that, for non-shortest paths, the correlations cannot grow too rapidly, with their growth controlled by the excess lengths of the two paths.

\begin{prop} 
\label{p_c}
Let $d\ge 2$ and $\alpha\in (d,2d)$. 
Almost surely on the event $\mathcal{E}_t^{\ssup 1}$, 
the following holds. 
\begin{itemize}
\item[(1)] As $t\to\infty$, uniformly for all $y, \hat y\in \PP$
\begin{align}
\label{shortc}
\Et \big(H_t(y)H_t(\hat y)\big)=o(1).
\end{align}

\item[(2)] For all sufficiently large $t$, for any $(m,w,s), (\hat m, \hat w, \hat s)$ and
any $y\in \mathcal{P}_{t, m, w, s}, \hat y\in \mathcal{P}_{t, \hat m, \hat w, \hat s}$
\begin{align}
\label{longc}
\big|\Et \big(H_t(y)H_t(\hat y)\big)\big|\le f_t^{-9(2w+s+2\hat w + \hat s)}.
\end{align}
\end{itemize}
\end{prop}

\begin{proof} 
For any $y \in \mathcal{P}_{all}$ denote by 
\begin{align*}
L_y(x)=|\{i: y_i=x\}|,  \qquad x\in \Z^d,
\end{align*}
its local time and observe that 
\begin{align*}
H_t(y)
=\exp\Big\{\sum_{x\neq Z_t}L_y(x)\Big(\beta_t \big[\zeta_t(x)-\mu_t(x)\big]
-\lambda_{t,x}(\beta_t)\Big)\Big\}-1. 
\end{align*}
Hence for any two paths $y, \hat y\in \mathcal{P}_{all}$ we have 
\begin{align}
\Et \big(H_t(y)H_t(\hat y)\big)
&=\exp\Big\{\sum_{x\neq Z_t}\Lambda_{t, x}(L_y(x)+ L_{\hat y}(x))\Big\}
-\exp\Big\{\sum_{x\neq Z_t}\Lambda_{t, x}(L_y(x))\Big\}\notag\\
&\phantom{aaaaaaaaaaaaaaaaaaaaaaaaaaaaa}-\exp\Big\{\sum_{x\neq Z_t}\Lambda_{t, x}(L_{\hat y}(x))\Big\}
+1\notag\\
&=e^{\Gamma_t(y)+\Gamma_t(\hat y)}\big(e^{\Delta_t(y, \hat y)}-1\big)
+\big(e^{\Gamma_t(y)}-1\big)\big(e^{\Gamma_t(\hat y)}-1\big),
\label{hh}
\end{align}
where 
\begin{align*}
\Delta_t(y, \hat y)=\sum_{x\neq Z_t}\big[\Lambda_{t,x}(L_y(x)+L_{\hat y}(x))
-\Lambda_{t,x}(L_y(x))-\Lambda_{t,x}(L_{\hat y}(x))\big]
\end{align*}
and 
\begin{align*}
\Gamma_t(y)=\sum_{x\neq Z_t}\Lambda_{t,x}(L_y(x))
\qquad\text{and}\qquad
\Gamma_t(\hat y)=\sum_{x\neq Z_t}\Lambda_{t,x}(L_{\hat y}(x)).
\end{align*}

(1) To prove~\eqref{shortc} consider $y, \hat y\in \PP$ and observe that  
$L_y(x), L_{\hat y}(x)\in \{0,1\}$ for all $x$. Since 
$\Lambda_{t,x}(0)=\Lambda_{t,x}(1)=0$ we obtain 
\begin{align*}
\Gamma_t(y)=\Gamma_t(\hat y)=0. 
\end{align*}
Using Lemmas~\ref{l_b2z} and~\ref{l_lll} 
we get
\begin{align*}
|\Delta_t(y, \hat y)|=\Big|\!\!\!\!\!\!\!\!\sum_{\substack{x\neq Z_t\\ L_y(x)=L_{\hat y}(x)=1}}
\!\!\!\!\!\!\!\!\Lambda_{t,x}(2)
\Big|
\le |Z_t| O(\beta_t^2)=o(1),
\end{align*}
which implies 
\begin{align*}
\Et (H_t(y)H_t(\hat y))=o(1).
\end{align*}

(2) To prove~\eqref{longc}
consider $y\in \mathcal{P}_{t,m,w,s}$ and $\hat y\in \mathcal{P}_{t,\hat m,\hat w,\hat s}$. 
We may assume that $2w+s+2\hat w+\hat s\neq 0$ as otherwise the statement follows from (1).
Using $\Lambda_{t,x}(0)=\Lambda_{t,x}(1)=0$ and Lemma~\ref{l_lll}
we obtain 
\begin{align}
\label{ga1}
|\Gamma_t(y)|\le \sum_{\substack{x\neq Z_t\\ L_y(x)\ge 2}}|\Lambda_{t,x}(L_y(x))|
\le \log (1/f_t)\sum_{\substack{x\neq Z_t\\ L_y(x)\ge 2}}L_y(x)\le 2(2w+s) \log (1/f_t)
\end{align}
and similarly 
\begin{align}
\label{ga2}
|\Gamma_t(\hat y)|\le 2(2\hat w+\hat s) \log (1/f_t). 
\end{align}
Further, using Lemmas~\ref{l_b2z} and~\ref{l_lll} together with~\eqref{ga1} and~\eqref{ga2} we obtain
\begin{align*}
|\Delta_t (y, \hat y)|
&\le \Big|\!\!\!\!\!\!\!\!\sum_{\substack{x\neq Z_t\\ L_y(x)=L_{\hat y}(x)=1}}\!\!\!\!\!\!\!\!\Lambda_{t,x}(2)\Big|
+\Big|\!\!\!\!\!\!\!\! \sum_{\substack{x\neq Z_t\\ L_y(x)\ge 2\text{ or }  L_{\hat y}(x)\ge 2}}\!\!\!\!\!\!\!\!
\Lambda_{t,x}(L_y(x)+L_{\hat y}(x))\Big|
+|\Gamma_t(y)|+|\Gamma_t(\hat y)|\\
&\le \min\{\ell(y)-m, \ell(\hat y)-\hat m\}O(\beta_t^2)\\
&+\log(1/f_t)\!\!\!\!\!\!\!\!\!\!\!\!
\sum_{\substack{x\neq Z_t\\ L_y(x)\ge 2\text{ or } L_{\hat y}(x)\ge 2}}\!\!\!\!\!\!\!\!\!\!\!\!
\big[L_y(x)+L_{\hat y}(x)\big]
+|\Gamma_t(y)|+|\Gamma_t(\hat y)|\\
&\le |Z_t| O(\beta_t^2)+(2w+s+2\hat w+\hat s)(O(\beta_t^2)+5\log(1/f_t))
\le 6\log(1/f_t)(2w+s+2\hat w+\hat s).
\end{align*}

Substituting this into~\eqref{hh} we obtain 
\begin{align*}
\big|\Et (H_t(y)H_t(\hat y))\big|\le f_t^{-9(2w+s+2\hat w +\hat s)}
\end{align*}
as required. 
\end{proof}

\begin{proof}[Proof of Theorem~\ref{th_d}] 
By Proposition~\ref{p_xixi} it suffices to prove that $\Xi_t^{\rm min}\to 0$ and $\Xi_t\to 0$
in probability, since 
$\text{P}(\mathcal{E}_t^{\rm loc})\to 1$ 
by~\eqref{comploc0}, 
$\text{\rm P}(\mathcal{E}_t^{\ssup 1}\cap \mathcal{E}_t^{\ssup 2})\to 1$ by 
Proposition~\ref{events} and $\text{\rm P}(Z_t=Z_t^*)\to 1$
by Proposition~\ref{zz}. 
\smallskip

By Proposition~\ref{p_c} and 
since $\mathcal{E}_t^{\ssup 1}\in \mathcal{G}_t$ we have 
\begin{align}
\label{l2conv}
\text{E}\Big[ (\Xi_t^{\rm min})^2\1_{\mathcal{E}_t^{\ssup 1}}\Big]
= \text{E} \Big[\1_{\mathcal{E}_t^{\ssup 1}} \Et (\Xi_t^{\rm min})^2\Big]=
\text{E}\Big[\1_{\mathcal{E}_t^{\ssup 1}} \frac{1}{N(Z_t)^2}\sum_{y, \hat y\in \PP}
\Et (H_t(y)H_t(\hat y))\Big]=o(1)
\end{align}
since $o(1)$ in Proposition~\ref{p_c} according to our notation is uniform on the probability space. 
This implies $\Xi_t^{\rm min}\to 0$ in probability.
\smallskip

In order to prove that $\Xi_t\to 0$, by Proposition~\ref{p_xi} it suffices to show that the main term 
of the upper bound~\eqref{upbo} converges in distribution to zero. 
Similarly to~\eqref{l2conv} we will show the stronger $L^2$-type convergence on $\mathcal{E}_t^{\ssup 1}$.  
By Proposition~\ref{p_c}, Lemma~\ref{l_npaths} and 
since $\mathcal{E}_t^{\ssup 1}\in \mathcal{G}_t$ we have
\begin{align*}
\text{E} \Big[\Big(
&\frac{2} {N(Z_t)}\sum_{(m,w,s)\neq 0} \sum_{y\in \mathcal{P}_{t, m, w, s}}
  \frac{t^m}{ \xi(Z_t)^{2w+s}m!}H_t(y)
\Big)^2\1_{\mathcal{E}_t^{\ssup 1}}\Big]\\
&\le \text{E} \Big[\1_{\mathcal{E}_t^{\ssup 1}} 
\frac{4}{N(Z_t)^2}\sum_{\substack{(m,w,s)\neq 0\\ (\hat m,\hat w,\hat s)\neq 0}}
\sum_{\substack{y\in\mathcal{P}_{t, m, w, s}\\ \hat y\in\mathcal{P}_{t, \hat m, \hat w, \hat s}}}
\frac{t^{m+\hat m}}{\xi(Z_t)^{2w+s+2\hat w+\hat s}m!\hat m!}\big|\Et (H_t(y)H_t(\hat y))\big|
\Big]\\
&\le \text{E} \Big[\1_{\mathcal{E}_t^{\ssup 1}} 
\frac{4}{N(Z_t)^2}\sum_{\substack{(m,w,s)\neq 0\\ (\hat m,\hat w,\hat s)\neq 0}}
\frac{t^{m+\hat m}}{(f_t^9\xi(Z_t))^{2w+s+2\hat w+\hat s}m!\hat m!}|\mathcal{P}_{t, m, w, s}|
|\mathcal{P}_{t, \hat m, \hat w, \hat s}|
\Big]\\
&=\text{E} \Big[\1_{\mathcal{E}_t^{\ssup 1}} \Big(
 \frac{2}{N(Z_t)}\sum_{(m,w,s)\neq 0}
\frac{t^{m}}{(f_t^9 \xi(Z_t))^{2w+s}m!}
|\mathcal{P}_{t, m, w, s}|
\Big)^2\Big]=o(1)
\end{align*}
since $o(1)$ in Lemma~\ref{l_npaths}  is uniform on the probability space. This concludes the proof of 
 $\Xi_t\to 0$ in probability.
\end{proof}
\bigskip


\section{Appendix}

In this final section we collect the scaling limit results relying on the standard 
point process techniques. 
\smallskip

It was shown in~\cite[Lemma 6.2]{klms} that, as $t\to\infty$, 
\begin{align}
\label{four}
\Big(\frac{Z_t^*}{r_t}, \frac{\hat Z_t^*}{r_t}, \frac{\Psi_t^*(Z_t^*)}{a_t}, \frac{\Psi_t^*(\hat Z_t^*)}{a_t}\Big)
\Rightarrow (X, \hat X, \Upsilon, \hat\Upsilon),
\end{align}
and the limit random variable has density 
\begin{align}
\label{dens_p4}
p(x, \hat x, w, \hat w)=\frac{\alpha^2\exp\{-\theta \hat w^{d-\alpha}\}}{(w+q|x|)^{\alpha+1}
(\hat w+q|\hat x|)^{\alpha+1}}\1_{\{w> \hat w>0\}}, 
\qquad\text{with}\quad
\theta=\frac{2^d B(\alpha-d, d)}{q^d (d-1)!}, 
\end{align}
where $B$ denotes the Beta function.
In particular, 
\begin{align}
\label{two}
\Big(\frac{Z_t^*}{r_t}, \frac{\Psi_t^*(Z_t^*)}{a_t}\Big)
\Rightarrow (X, \Upsilon),
\end{align}
and the limit random variable has density
\begin{align}
\label{dens_p2}
p(x, w)=\frac{\alpha\exp\{-\theta w^{d-\alpha}\}}{(w+q|x|)^{\alpha+1}}\1_{\{w>0\}}, 
\end{align}
which is obtained from~\eqref{dens_p4} by integrating over $\hat x$ and $\hat w$.  
\smallskip

The key ingredient of the proof was the family of point  processes  
\begin{align*}
\Pi_s=\sum_{z\in \Z^d}\delta\Big(\frac z s, \frac{\xi(z)}{s^{d/\alpha}}\Big),\qquad s>0,
\end{align*}
where $\delta(x,y)$ denotes the Dirac measure at $(x,y)$. A straightforward generalisation of the proof of~\cite[Prop.\ 3.21]{reznick}
implies that, as $s\to\infty$,  
\begin{align}
\label{ppp}
\Pi_s\Rightarrow \Pi, 
\end{align} 
where $\Pi$ is 
a Poisson point process on the state space $\R^d\times (0,\infty]$ with the intensity measure 
\begin{align*}
\nu(dx\otimes dy)=dx\otimes \frac{\alpha}{y^{\alpha+1}}dy. 
\end{align*}
Denote the corresponding probability by $\Pp$. 
\smallskip

In the next proposition we identify the scaling limit of $(Z_t, \Psi_t(Z_t), \xi(Z_t))$ using a similar approach. In particular, 
we show that the scaling limit of $(Z_t, \Psi_t(Z_t))$ coincides with that of 
$(Z^*_t, \Psi_t^*(Z^*_t))$ given in~\eqref{two}. An analogous statement also holds for the joint convergence in~\eqref{four}, but we shall not need it, since it suffices to control the gap between the largest and second largest values of one of the two functionals.

\begin{prop} 
\label{p_new}
As $t\to\infty$, 
\begin{align*}
(X_t, \Upsilon_t, Y_t)\Rightarrow (X, \Upsilon, \Upsilon+q|X|).
\end{align*}
\end{prop}

\begin{proof} 
Due to the deterministic relation $Y_t=\Upsilon_t+q|X_t|$ following from~\eqref{ra}, 
it suffices to show that $(X_t, Y_t)\Rightarrow (X, \Upsilon+q|X|)$.
\smallskip

Let $A$ be a compact Borel subset of $\{(x,y): x\in \R^d, y>q|x|\}$ such that its boundary has zero Lebesgue 
measure. Hence $A\subset \{(x,y): x\in\R^d, y>q|x|+\e\}$ with some $\e>0$, and
\begin{align*}
\nu\big(\{(x,y): y>q|x|+\e\}\big)=\int_{\R^d}\frac{dx}{(q|x|+\e)^{\alpha}}
=\frac{2^d}{(d-1)!}\int_0^{\infty}\frac{t^{d-1} dt}{(qt+\e)^{\alpha}}<\infty. 
\end{align*}
It follows from $Z_t$ being the maximizer of $\Psi_t$ and from 
\begin{align}
\label{ppppsi}
\frac{\Psi_t(z)}{a_t}=\frac{\xi(z)}{a_t}-q\frac{|z|}{r_t},
\end{align}
that 
\begin{align*}
\text{\rm P}\Big(\Big(\frac{Z_t}{r_t}, \frac{\xi(Z_t)}{a_t}\Big)\in A\Big)
=\int_A \text{\rm P}\Big(\Pi_{r_t}(dx\times dy)=1, \Pi_{r_t}\big(\big\{(u,v): v-q|u|>y-q|x|\big\}\big)=0\Big).
\end{align*} 
Using convergence~\eqref{ppp} we then get 
\begin{align}
\label{lim}
\lim_{t\to\infty}\text{\rm P}\big((X_t, Y_t)\in A\big)
&=\int_A \Pp\Big(\Pi(dx\times dy)=1, \Pi\big(\big\{(u,v): v-q|u|>y-q|x|\big\}\big)=0\Big).
\end{align}
Using a calculation from the proof of~\cite[Prop.\ 3.8]{hms}, we compute 
\begin{align*}
\nu\big(\big\{(u,v): v-q|u|>y-q|x|\big\}\big)
&=\int_{\R^d}\int_{y+q(|u|-|x|)}^{\infty}\frac{\alpha dvdu}{v^{\alpha+1}}\\
&=\int_{\R^d}\frac{du}{(y+q(|u|-|x|))^{\alpha}}
=\theta(y-q|x|)^{d-\alpha}. 
\end{align*}
Substituting this into~\eqref{lim} we obtain 
\begin{align*}
\lim_{t\to\infty}\text{\rm P}\big((X_t, Y_t)\in A\big)
=\int_A \exp\big\{-\theta(y-q|x|)^{d-\alpha}\big\}\frac{\alpha dxdy}{y^{\alpha+1}}
=\int_A p(x,y-q|x|)dxdy.
\end{align*}
Finally, the collection of test sets $A$ 
is big enough as according to~\eqref{dens_p2}
the integral of $p(x,y-q|x|)$ over $\{(x,y): x\in\R^d, y>q|x|\}$  equals one. 
\end{proof}

In the next proposition we show that $Z_t=Z_t^*$ and $\Psi_t(Z_t)\sim \Psi_t^*(Z_t^*)$ with probability tending to one. The key point is that both $Z_t$ and $Z_t^*$ are now known to be of order $r_t$, and for sites on this scale the difference $\Psi_t(z)-\Psi_t^*(z)$ is of order $o(a_t)$, since $\Psi_t$ was chosen to approximate $\Psi_t^*$ closely enough. On the other hand, the gap between the largest and second largest values of $\Psi_t^*$ is of order $a_t$. The discrepancy between the two functionals is therefore negligible compared with this gap, which forces their maximisers to coincide.

\begin{prop} 
\label{zz}
$\phantom{a}$

\begin{itemize}
\item[(1)] As $t\to\infty$,  $\text{\rm P}(Z_t=Z_t^*)\to 1$. 
\item[(2)] There is an event $\mathcal{E}_t^*$  such that $\Psi_t(Z_t)=(1+o(1))\Psi_t^*(Z_t^*)$
almost surely on $\mathcal{E}_t^*$ and $\text{\rm P}(\mathcal{E}_t^*)\to 1$. 
\end{itemize}
\end{prop}

\begin{proof} 
Let 
\begin{align*}
\mathcal{E}_t^*=\Big\{f_t<\frac{|Z_t|}{r_t}<g_t, f_t<\frac{|Z_t^*|}{r_t}<g_t, 
f_t<\frac{\Psi_t(Z_t)}{a_t}<g_t, f_t<\frac{\Psi_t^*(Z_t^*)}{a_t}< g_t, \frac{\Psi_t^*(Z_t^*)-\Psi_t^*(\hat Z_t^*)}{a_t}> f_t\Big\}
\end{align*}
and observe that $\text{\rm P}(\mathcal{E}_t^*)\to 1$ by~\eqref{four}, \eqref{dens_p4} and 
Proposition~\ref{p_new}.
\smallskip

(1) It suffices to show that $\mathcal{E}_t^*\subset \{Z_t=Z_t^*\}$.  
Consider $\mathcal{E}_t^*$ and suppose $Z_t\neq Z_t^*$.  Then $Z_t$ is at best the second largest 
value of $\Psi_t^*$ and hence
\begin{align}
\label{ap3}
\frac{\Psi_t^*(Z_t^*)-\Psi_t^*(Z_t)}{a_t}\ge \frac{\Psi_t^*(Z_t^*)-\Psi_t^*(\hat Z_t^*)}{a_t}> f_t.
\end{align}

On the other hand, the indicator function condition in the definition of $\Psi_t^*$ is satisfied at $Z_t$ 
since  
\begin{align*}
t\xi(Z_t)=t\Psi_t(Z_t)+|Z_t|q\log t>ta_tf_t+qr_t f_t\log t>r_t g_t >|Z_t|, 
\end{align*}
implying that the values of both functionals $\Psi_t^*$ and $\Psi_t$ are close at each of the points 
$Z_t$ and $Z_t^*$, namely,    
\begin{align}
\label{ap1}
\frac{|\Psi_t(Z_t)-\Psi_t^*(Z_t)|}{a_t}&
=\frac{1}{\log t}\Big|\frac{|Z_t|}{r_t}\log \frac{\xi(Z_t)}{a_t}
-q\frac{|Z_t|}{r_t}\log\log t-\frac{\log N(Z_t)}{r_t}\Big|\notag\\
&=O\Big(g_t\frac{\log\log t}{\log t}\Big)=o(f_t)
\end{align}
since $N(z)\le (2d)^{|z|}$ for any $z$, and similarly
\begin{align}
\label{ap2}
\frac{|\Psi_t(Z_t^*)-\Psi_t^*(Z_t^*)|}{a_t}=o(f_t).
\end{align}
Now it follows from~\eqref{ap1} and~\eqref{ap2} that 
\begin{align*}
\frac{\Psi_t^*(Z_t^*)-\Psi_t^*(Z_t)}{a_t}=\frac{\Psi_t(Z_t^*)-\Psi_t(Z_t)}{a_t}+o(f_t)\le o(f_t)
\end{align*}
since $Z_t$ is the maximiser of $\Psi_t$, thus contradicting~\eqref{ap3}. 
\smallskip

(2) To prove the second statement, we use $Z_t=Z_t^*$ on $\mathcal{E}_t^*$ together with~\eqref{ap1} to obtain 
\begin{align*}
\frac{|\Psi_t(Z_t)-\Psi_t^*(Z_t^*)|}{\Psi_t^*(Z_t^*)}< \frac{|\Psi_t(Z_t)-\Psi_t^*(Z_t)|}{a_t f_t}
=O\Big(\frac{g_t}{f_t}\frac{\log\log t}{\log t}\Big)=o(1)
\end{align*}
as required. 
\end{proof}

In the final proposition, we verify that all the events introduced as typical indeed occur with probability either tending to $1$ or arbitrarily close to $1$. For the events $\mathcal{E}_t^{\ssup 1}$ and $\mathcal{E}_{t,\delta}$, this follows directly from the limiting distribution of $(X,Y)$ and the positivity of its density. For the event $\mathcal{E}_t^{\ssup 2}$, however, we need to dig a little deeper into the point process convergence~\eqref{ppp}.

\begin{prop}
\label{events} $\phantom{a}$
\begin{itemize}
\item[(1)] $\text{\rm P}(\mathcal{E}_t^{\ssup 1})\to 1$ and $\text{\rm P}(\mathcal{E}_t^{\ssup 2})\to 1$ as $t\to \infty$.
\item[(2)] For any $\e>0$ there is $\delta>0$ such that 
$\text{\rm P}(\mathcal{E}_{t, \delta})>1-\e$ for all sufficiently large $t$. 
\end{itemize} 
\end{prop}

\begin{proof} (1) Consider the event $\mathcal{E}_t^{\ssup 1}$. By Proposition~\ref{p_new}
the pair 
$(X_t, Y_t)$ converges in distribution, which implies that the first two conditions in the definition of $\mathcal{E}_t^{\ssup 1}$ hold with probability tending to $1$. The fact that the limiting random variable has a density further implies the same for the third condition. 
Since $Y>q|X|$ almost surely  the probability of the fourth condition also 
tends to $1$. 
\smallskip

Now consider the event $\mathcal{E}_t^{\ssup 2}$. 
Let $\e>0$ 
and observe, using the same compactness argument as in Proposition~\ref{p_new}, that 
\begin{align*}
\text{\rm P}
\Big(\max_{\substack{|z|<|Z_t|(1+\e)\\ z\neq Z_t}}\frac{\xi(z)}{\xi(Z_t)}< 1-\e\Big)
&=\text{\rm P}\Big(\max_{\substack{|\frac{z}{r_t}|<|\frac{Z_t}{r_t}|(1+\e)\\ z\neq Z_t}}\frac{\xi(z)}{a_t}<  
(1-\e)\frac{\xi(Z_t)}{a_t}\Big)\\
&\ge \int_{\{(x,y): y>q|x|+\e\}} \text{\rm P}\Big(\Pi_{r_t}(dx\times dy)=1, 
\Pi_{r_t}\big(A_{x,y}^{\e}\big)=0
\Big), 
\end{align*}
where 
\begin{align*}
A_{x,y}^{\e}=\big\{ (u,v): v-q|u|>y-q|x|\big\} \cup \big\{ (u,v): |u|\le (1+\e)|x|, v\ge (1-\e)y\big\}.
\end{align*}

Using convergence~\eqref{ppp} we get 
\begin{align*}
\liminf_{t\to\infty}\text{\rm P}
\Big(\max_{\substack{|z|<|Z_t|(1+\e)\\ z\neq Z_t}}\frac{\xi(z)}{\xi(Z_t)}< 1-\e\Big)
&\ge 
\int_{\{(x,y): y>q|x|+\e\}} \Pp\Big(\Pi(dx\times dy)=1, 
\Pi\big(A_{x,y}^{\e}\big)=0
\Big).
\end{align*}
It is easy to see that the sets $A_{x,y}^{\e}$ decrease to the set $\big\{ (u,v): v-q|u|>y-q|x|\big\}$
as $\e\downarrow 0$. Hence due to the monotone convergence we have 
\begin{align}
\lim_{\e\downarrow 0}\ 
\liminf_{t\to\infty}\text{\rm P}
&\Big(\max_{\substack{|z|<|Z_t|(1+\e)\\ z\neq Z_t}}\frac{\xi(z)}{\xi(Z_t)}< 1-\e\Big)\notag\\
&\ge \int_{\{(x,y): y>q|x|\}} \Pp\Big(\Pi(dx\times dy)=1, \Pi\big(\big\{(u,v): v-q|u|>y-q|x|\big\}\big)=0\Big)=1
\label{limlim}
\end{align}
as according to~\eqref{lim} this is just the probability of $Y>q|X|$, which is 
equal to one. It remains to observe that
\begin{align*}
\liminf_{t\to\infty}\text{\rm P}\Big(\max_{\substack{|z|<R_t\\ z\neq Z_t}}\frac{\xi(z)}{\xi(Z_t)}< 1-f_t\Big)
\ge \liminf_{t\to\infty}\text{\rm P}
\Big(\max_{\substack{|z|<|Z_t|(1+\e)\\ z\neq Z_t}}\frac{\xi(z)}{\xi(Z_t)}< 1-\e\Big)
\end{align*}
for any $\e>0$, which together with~\eqref{limlim} implies  that $\text{\rm P}\big(\mathcal{E}_t^{\ssup 2}\big)\to 1$.
\smallskip

(2) This follows immediately from Proposition~\ref{p_new}.   
\end{proof}


\bigskip


\begin{thebibliography}{99}

\bibitem{al} A.\ Auffinger, O.\ Louidor, Directed polymers in a random environment with heavy tails, \emph{Comm.\ Pure Appl.\ Math.} 64, 183--204, 2011

\bibitem{ap} E.\ Archer, A.\ Pein, Parabolic Anderson model on critical Galton--Watson trees in a Pareto environment, \emph{Stochastic Process.\ Appl.} 159, 34--100, 2023

\bibitem{bks} M.\ Biskup, W.\ K\"onig, R.S.\ dos Santos, Mass concentration and aging in the parabolic Anderson model with doubly-exponential tails, \emph{Probab.\ Theory Relat.\ Fields} 171, 251--331, 2018

\bibitem{nikos} P.S.\ Dey and N.\ Zygouras, High temperature limits for $(1+1)$-dimensional 
directed polymer with heavy-tailed disorder, \emph{Ann.~Probab.} {\bf 44}, 4006--4048, 2016


\bibitem{fm} A.\ Fiodorov, S.\ Muirhead, Complete localisation and exponential shape of the parabolic Anderson model with Weibull potential field, \emph{Electron.\ J.\ Probab.} 19 (58), 1--27, 2014


\bibitem{gm1} J.\ G\" artner, S.A.\ Molchanov, Parabolic problems for the Anderson model. I. Intermittency
and related topics, \emph{Comm.\ Math.\ Phys.} 132.3, 613--655, 1990

\bibitem{hms} R.\ v.d.Hofstad, P.\ M\" orters, N.\ Sidorova,
Weak and almost sure limits for the parabolic Anderson model with heavy-tailed potential,
\emph{Ann.\ Appl.\ Probab.}, 18, No.\ 6, 2450--2494, 2008


\bibitem{k} W.\ K\"onig, The parabolic Anderson model: random walk in random potential, \emph{Birkh\"auser}, 2016

\bibitem{klms} W.\ K\" onig, H.\ Lacoin, P.\ M\" orters, N.\ Sidorova, A two cities theorem for the parabolic Anderson model, \emph{Ann.\ Probab.} 37 (1) 347--392, 2009

\bibitem{kms} W.\ K\" onig, P.\ M\" orters, N.\ Sidorova, 
Complete localisation in the parabolic Anderson model with heavy-tailed potential, arXiv:math/0608544, 
2006

\bibitem{mps1} S.\ Muirhead, R.\ Pymar, N.\ Sidorova, Delocalising the parabolic Anderson model 
through partial duplication of the potential, \emph{Probab.\ Theory Relat.\ Fields} 171, 917--979, 2018 

\bibitem{mps2} S.\ Muirhead, R.\ Pymar, N.\ Sidorova, A new phase transition in the parabolic Anderson model with partially duplicated potential, \emph{Stochastic Process.\ Appl.} 129 (11)
4704--4746, 2019

\bibitem{or1} M.\ Ortgiese, M.I.\ Roberts, Intermittency for branching random walk in Pareto environment, \emph{Ann.\ Probab.} 44, 2198--2263, 2016

\bibitem{or2} M.\ Ortgiese, M.I.\ Roberts, One-point localization for branching random walk in Pareto environment, 
\emph{Electron.\ J.\ Probab.} 22, Paper No.\ 6, 1--20, 2017

\bibitem{reznick} S.\ I.\ Resnick, Extreme Values, Regular Variation, and Point Processes, 
\emph{Applied Probability. A Series of the Applied Probability Trust} 4, Springer, New York, 1987

\bibitem{st} N.\ Sidorova, A.\ Twarowski, Localisation and ageing in the parabolic Anderson model with Weibull potential, \emph{Ann.\ Probab.} 42 (4), 1666--1698, 2014

\end{thebibliography}
\end{document}